\documentclass[11pt,reqno]{amsart}
 \usepackage[colorlinks=true, bookmarks=true]{hyperref}
   \hypersetup{colorlinks, citecolor=blue, filecolor=blue, linkcolor=red, urlcolor=blue}

\usepackage{proof}
\usepackage{float}
\usepackage{lipsum}
\usepackage{graphicx}
\usepackage{mathtools}
\usepackage{amsfonts,mathrsfs}
\usepackage{amssymb}
\usepackage{array}
\usepackage{thmtools, thm-restate}

\newcolumntype{M}[1]{>{\centering\arraybackslash}m{#1}}
\newcolumntype{N}{@{}m{0pt}@{}}

\newcommand{\fm}{\mathrm{Fm}}
\newcommand{\f}{\varphi}
\newcommand{\p}{\psi}
\renewcommand{\a}{\alpha}

\renewcommand{\c}{\gamma}
\renewcommand{\d}{\delta}
\newcommand{\ra}{\rightarrow}

\newcommand{\nm}{{\q}}

\newcommand{\rj}{\nm_{RJ}}
\newcommand{\urj}{\nm_{\urjt}}
\newcommand{\prj}{\nm_{\prjt}}
\newcommand{\arj}{\nm_{\arjt}}

\newcommand{\urjt}{uRJ}
\newcommand{\prjt}{\textit{l}RJ}
\newcommand{\arjt}{\ensuremath{\a}RJ}

\newcommand{\rp}{r^{l}}
\newcommand{\ru}{r^{u}}
\newcommand{\rfu}[1]{r^{u_{#1}}}

\newcommand{\rmon}{\text{RMO}}
\newcommand{\umon}{\text{UMO}}

\newcommand{\multid}{\mathcal{M}_{\D}}

\renewcommand{\H}{\mathcal{H}}
\newcommand{\D}{\mathcal{D}}

\usepackage{color,graphicx,amsmath,wrapfig}
\usepackage[all]{xypic}

\newcommand{\maxh}{\hat{\mathcal{H}}}
\newcommand{\fin}[1]{\H_{#1}^{f}}
\newcommand{\ol}[1]{\overline{#1}}
\renewcommand{\H}{\mathcal{H}}
\newcommand{\luka}{\L u\-ka\-si\-e\-w\-icz}

\usepackage{float}
\usepackage{pgf,tikz}
\usepackage{amsmath}

\usepackage[utf8]{inputenc}
\theoremstyle{plain}

\newtheorem{lemma}{Lemma}
\newtheorem{corollary}{Corollary}
\newtheorem{proposition}{Proposition}

\theoremstyle{definition}
\newtheorem{definition}{Definition}

\newtheorem{example}{Example}

\theoremstyle{remark}   
\newtheorem{remark}{Remark}

\renewcommand{\phi}{\varphi}
\renewcommand{\epsilon}{\varepsilon}
\newcommand{\q}{\mid\hspace{-2.4pt}\sim}

 \DeclareMathOperator*{\argmin}{arg\,min}

 \def\checkmark{\tikz\fill[scale=0.4](0,.35) -- (.25,0) -- (1,.7) -- (.25,.15) -- cycle;}

\usepackage{easyReview}

\begin{document}
%title{Logical systems for reasoning with \replace{scientific}{statistical}  hypotheses
\title{A logical framework for data-driven reasoning}

\author{Paolo Baldi, Esther Anna Corsi, Hykel Hosni}
\maketitle

\begin{abstract}
We introduce and investigate a family of consequence relations with the goal of capturing certain important patterns of data-driven inference. The inspiring idea for our framework is the fact that data may reject, possibly to some degree, and possibly by mistake, any given scientific hypothesis.  There is no general agreement in science about how to do this, which motivates putting forward a logical formulation of the problem.  We do so by  investigating distinct definitions of ``rejection degrees'' each yielding a consequence relation. Our investigation leads to novel variations on the theme of \textit{rational consequence relations}, prominent among non-monotonic logics.
\end{abstract}


\section{Introduction and motivation}
The research reported in this note originates in \cite{Baldi2023}, where a case is made for investigating  consequence relations capable of expressing certain aspects of scientific inference. In particular, we are interested in capturing the fact that data may lead scientists to reject, possibly to some degree, and possibly  by mistake, any given scientific hypothesis.

The canonical treatment of this general problem makes one of its first appearances in a 1925 book which went on to shape the methodology of much empirical science \cite{Fisher1954}. In it, R.A. Fisher draws attention to  ``the logical nature" of the inference underpinning \emph{tests of significance}.  The problem he sets out to address is that of examining a scientific hypothesis ($H_0$) based on the observations it leads to, if true. When the observations so obtained are improbable enough, they provide grounds for us to reject the assumption that $H_0$ is indeed true.  Fisher's  line of reasoning culminates with what became known as the \emph{Fisher disjunction}:
\begin{quote}
The force with which such a conclusion is supported is logically that of the simple disjunction: \textit{Either} an exceptionally rare event has occurred, \textit{or}  [$H_0$]
%the theory of random distribution 
is not true. (\cite{Fisher1956}, p.39.)
\end{quote}
We understand Fisher's ``simple" as ``being governed by logic", which of course for him could only be (an informal version of) classical logic. Setting aside some rare yet notable exceptions discussed in Section \ref{sec:related}, the statistical and methodological communities seem to have taken this informal-classical-logic  view at face value. Interestingly, this applies to both supporters and critics of the procedures put forward by Fisher. For the focus of the long-standing debates on tests of significance, and in particular on the infamous p-value \cite{Wasserstein2016}, is usually on the meaning and nature of probability, rather than on the properties of the inferences scientists make with it. 

This paper takes a logical perspective on inference based on the data-driven rejection of scientific hypotheses. And, insofar as this is possible, it aims to do so independently of any philosophical view on probability. We ask which properties are desirable for a consequence relation whose intended semantics is based on the degree of rejection that a given set of data provides to a given hypothesis.  Hence we consider a more abstract and general problem than that envisaged by Fisher (and his many followers), which is however recovered as a special case of our question. Our results suggest that reasonable answers to our main question will be variations on the theme of non-monotonic consequence relation, in the sense brought to the attention of logicians by \cite{Shoham1987,Kraus1990,Lehmann1992,Makinson2005}.

The remainder of this introductory section provides the essential background on the kind of scientific inference we focus on,  along with the required logical preliminaries. We clearly do not aim at exhausting the many ramifications of the vast topic of scientific inference. On the contrary, the following two subsections should be thought of as delimiting the scope of our investigation. We defer the discussion on related work to the concluding section of the paper.

\subsection{NHST}\label{sec:NHST} According to a basic tenet in the methodology of science, to assess a scientific hypothesis, one looks at its logical consequences. In this context, the pattern of inference captured by \emph{modus tollens} becomes prominent. If a sentence $\theta$, which is known to be false, follows logically from sentence $\phi$, then we should conclude $\neg\phi$. As we will now recall, this classical pattern of inference is often taken to lend its validity to the procedure known as Null Hypothesis Statistical Testing (NHST).

Let $H_0$ be a sentence (in classical logic)  standing for a scientific hypothesis, usually referred to as the \emph{null hypothesis}. The key idea in NHST is to set up an experiment which leads to observations under the assumption that $H_0$ is true. Denote by $\sigma$ a \emph{test statistics}, i.e. a function of the observations thus obtained. Finally, associate to this function the quantity usually referred to as the ``observed level of significance'', or the ``probability value'', or simply the ``p-value''. This latter is the calculated conditional probability of seeing equally extreme or more extreme data  (according to the test statistics) given $H_0$. In analogy with modus tollens one says that $H_0$ should be rejected if the p-value is small enough.  Here is a schematisation of the argument:
\begin{enumerate}
\item Suppose $H_0$;
\item Calculate the p-value for some test statistics $\sigma$ (i.e. a well-defined function of the data observed conditional on $H_0$);
\item The smaller the p-value, the stronger the reason to believe that $H_0$ is not true.
\end{enumerate}

Many statisticians and methodologists explicitly draw the parallel between (versions of) the above and modus tollens. To make a few notable examples, the authors of \cite{Dickson2011} speak of ``inductive modus tollens'', \cite{Mayo2018} refers to ``statistical modus tollens'', whereas \cite{Royall1997} tellingly notes that it is the analogy with modus tollens  which gives tests of significance  prominence in the ``scientific method'':
\begin{quote}
[Modus tollens] 
is at the heart of the philosophy of science, according to Popper. Its statistical manifestation is in [the] formulation of hypothesis testing that we will call ‘rejection trials’. (\cite{Royall1997}, p.72)
\end{quote}
Quite interestingly, this view is shared also by prominent critics of NHST, e.g. \cite{Szucs2017}.

Whilst the appeal of the loose analogy with modus tollens is clear, it is misleading nonetheless. For no probabilistic test will  deliver $\neg H_0$. Royall is again an authoritative voice who acknowledges the problem, but then argues that it is of no real consequence.``But the form of reasoning in the statistical version of the problem parallels that in deductive logic: if $H_0$ implies $E$ (with high probability), then not-$E$ justifies rejecting $H_0$'' \cite{Royall1997}, p.73.

The qualitative difference between $H_0$ being classically false and it being very improbable did not escape the attention of Fisher and his scrupulous followers, when they insist that the p-value is best interpreted as the degree to which observational data turns out to be incompatible with $H_0$.  So a significance test can only lead one to conclude that $H_0$ is \emph{unlikely to be true}, if the p-value is small enough. But then the validity of this conclusion owes to the Fisher disjunction, rather than to modus tollens. And in turn, the former is grounded on the metaphysical assumption that small-probability events normally do not happen. Many methodologists, probabilists and statisticians strongly disagree with this, as testified by the animosity of the long-lasting debate on this topic \cite{Howson1993,Wasserstein2016,Wasserstein2019}.

Since ``being unlikely true'' is logically distinct from ``being classically false'',  we have no reason to believe that the above schematisation holds under uncertainty. Indeed, it is well known that modus tollens need not be adequate for probabilistic reasoning \cite{Bosley1979}, and in general fails to deliver a point-valued probability \cite{Wagner2004}. Given that  uncertainty is the norm, rather than the exception in science, and that uncertainty is most typically quantified probabilistically, the standard justification for adopting the canonical significance tests is, at the very least, in clear need of a logical clarification.

Note that a line of criticism has indeed addressed ``the logic of NHST'', concluding that it is defective 
\cite{Pollard1987,Falk1995,Krueger2001,Szucs2017}. The following example, which probably appears for the first time in \cite{Pollard1987} is by now  standard. 
\begin{example}\label{ex:MT} Consider the following argument. 

  \begin{itemize}
    \item[1] Either Harold is a US citizen or he is not;
%      \hfill($H_0 \vee H_1$)
    \item[2] If Harold is a US citizen,  than he is most probably not a member      of Congress;
      %\hfill ($\delta_1$)
    \item[3] Harold is a member of Congress; %\hfill($\delta_2$) 
%    \item[4] $P(DATA+\mid H_0)=$ low \hfill
    \item[$\therefore$] Harold is likely not a US citizen. \end{itemize}
Whilst assumptions 1-3 are all true,  the conclusion is absurd, since  being a US citizen is a necessary condition to sit in the US Congress. And yet it would be arrived at through (a kind of) modus tollens. \end{example}

Those who use it take the argument in Example \ref{ex:MT} to be conclusive in showing the fallacious nature of NHST. Note however that the criticism is conclusive only if one assumes that classical logic is the logic against which a pattern of inference can be evaluated. However, as remarked above, classical logic is hardly adequate to capture the ``most probably'' and the ``likely'' which appear in 2 and in
{the conclusion}. Non-monotonic consequence relations, on the contrary can express them, as detailed in Section \ref{sec:MT}. And indeed, if one looks at it from the non-monotonic-logic point of view, a natural conclusion of premisses 1-3 in Example \ref{ex:MT} is that \emph{Harold is not a typical US citizen}.

Finally a remark on terminology. Statisticians and methodologists go to great lengths to distinguish significance tests from the associated decision of either ``retaining'' or ``rejecting'' the null hypothesis, a procedure made prominent by J. Neyman and E. Pearson. In this latter approach one fixes, before running the relevant experiment, a given threshold for rejection, usually referred to as ``critical region'' and denoted by $\alpha$. If, conditional on $H_0$, the observations fall within this region, then $H_0$ is rejected. However, the noticeable mathematical commonalities between the Fisher and the Neyman-Pearson takes on the problem, lead to the mishmash of the two methodologies under the heading NHST, see \cite{Chow1996} for a comprehensive analysis. In what follows we can be rather casual about this internal divide within the classical statistical tradition,  because none of our results below is affected by committing to either of the competing views.

\subsection{Strong Inference}\label{sec:SI}
The method of \emph{strong inference} was put forward in 1964 \emph{Science} editorial by mathematician and biophysicist John Platt \cite{Platt1964}. Ever since it has been argued, especially in the life sciences, to provide a methodological guidance to hypothesis testing for working scientists \cite{Kinraide2003} and indeed a sound generalisation of NHST \cite{Bookstein2014}.

Platt offers the following schematic representation of strong inference in \cite{Platt1964}:
\begin{quote}
\begin{enumerate}
    \item[1)] Devising alternative hypotheses; 
    \item[2)] Devising a crucial experiment (or several of them), with alternative possible outcomes, each of which will, as nearly as possible, exclude one or more of the hypotheses; 
    \item[3)] Carrying out the experiment so as to get a clean result;
    \item[1')] Recycling the procedure, making subhypotheses or sequential hypotheses to refine the possibilities that remain; and so on. 
    \end{enumerate} 
  \end{quote}
  According   to the author, the distinctive trait of {strong inference} lies in the fact that it requires  scientists to produce as many competing hypotheses as possible in the given context. This gives rise to a rich tree of alternatives. As strong inference proceeds, the tree is pruned by  (the logical consequences of) observations, i.e. outcomes of  ``crucial experiments". The process runs until, in the best possible scenario, a single hypothesis survives.

  Strong inference appears to have attracted very little interest from the logical community. However,  we suggested in \cite{Baldi2023} that, owing to the conceptual analogies with the  {Ulam-R\'{e}nyi} game, introduced by  Alfr{\'e}d R{\'e}nyi \cite{Renyi1961problem}  and Stanislaw Ulam \cite{ulam1991adventures},  strong inference can give us useful cues about the desirable properties of a consequence relation formalising data-driven inference.

\subsection{Non-monotonic consequence relations}\label{sec:nm} We will be using several symbols for related but distinct consequence relations. As usual, $\models$ denotes classical (propositional) consequence, and $\equiv$ classical equivalence. Then we use $\nm$ to denote an arbitrary non-classical consequence relation.

Given the revisable nature of scientific inference, our logical framework is closely related to the family of non-monotonic logics. Those originated in the 1980s, mainly within the artificial intelligence community, for knowledge representation and reasoning purposes, see \cite{Marquis2020} for a recent and broad appraisal with a comprehensive bibliography. After a couple of decades since the first proposals, the community of non-monotonic logicians converged on three important logical systems: \emph{cumulative} System C, the \emph{preferential} System P, and the \emph{rational} System R. Table \ref{tab:KLM} presents the latter.

\begin{table}[h!]
  \centering
  \begin{tabular}{cc}
    \infer[\text{(REF)}]{\f\q\f}{}  &\infer[\text{(RWE)}]{\f\q\xi} {\f \q \p & \p\models\xi} \\[2.0em]
    \infer[\text{(LLE)}]{\p\q\xi}{\f\q \xi & \f\models\p &\p\models\f} & 
                                                                         \infer[\text{(CMO)}] {\f\land\xi \q \p}{\f\q \p & \f\q\xi } \\[2.0em]
    \infer[ \text{(AND)}]{\f\q \p\land\xi} {\f\q \p & \f\q\xi}  & \infer[\text{(OR)}]{\f\lor                                                           \p\q \xi}{\f\q \xi& \p \q\xi}                                         \\[2.0em]
\multicolumn{2}{c}{\infer[\text{(RMO)}]{\f\land\xi\nm\p}{\f\nm\p & \f\not\nm\neg\p}}\\[2.0em]
   \end{tabular}
\caption{System R}
\label{tab:KLM}
\end{table}

System C and System P are two subsets of System R. %\footnote{\paolo{Mi confondeva un po' qui a prima lettura l'uso di former e latter (pensavo avesse a che fare con System R).Ho modificato, anche se è un po' meno elegante}} 
{System C} %The former 
is defined by (REF), (LLE), (RWE), and (CMO), whereas {System P} % the latter 
also satisfies (AND) and (OR). %\e{Note that (AND) is a derived rule in System C}. 
In the interest of readability,  we do not introduce distinct symbols for the three consequence relations characterizing those systems. %System C, System P, and System R.

\subsection{Plan of the paper} Section \ref{sec:quantitative} sets out the intended semantics for consequence relations based on ``degrees of rejection''. Its formalisation, denoted by $\rj$, will constitute a blueprint for defining consequence relations which capture specific aspects of this. Section \ref{sec:prj} investigates $\prj$ and $\arj$ which are based on probabilistic rejection and are closely related to maximum likelihood and NHST, respectively.
% is closely related to NHST. %, and is showed to satisfy all the properties of System P with the exception of {(OR)}. 
%This, as we will argue, owes to the probabilistic nature of semantics underpinning $\prj$. 
 Then, building closer ties with USI games, we introduce and investigate $\urj$ in Section \ref{sec:rrj}. This is shown 
 %to be a rational consequence relation in the sense of Table \ref{tab:KLM}, which in addition 
 satisfy (\umon),  a form  of constrained monotonicity which, to the best of our knowledge, has not been investigated before. Theorem \ref{thm:iminimal}, the main result of this paper, establishes a form of completeness for $\urj$. Section \ref{sec:conclusion} summarises our contributions and lists future works on this topic.

\section{The blueprint RJ-consequence}
\label{sec:quantitative}

The key contribution of this paper is a proposal for the formalisation of consequence relations capturing a central feature of scientific inference: rejecting certain hypotheses in the light of data. Doing this requires inevitable abstraction so,  to pin down the intended semantics of the consequence relations of interest, we first strip down our problem of many of the details encountered in scientific practice. This, we submit, leaves us with three central features of data-driven inference:  (i) There is a distinction between data and hypotheses, but it pertains to the attitude scientists have towards the statements representing them, rather than the nature of the statements themselves; (ii) experiments (the data-generating processes) can be repeated, and  may lead to non-unique outcomes; (iii) When making data-driven inferences, scientists can rely on knowledge that is not subject to questioning by the experiment at hand. Since Definition \ref{def:degrej} below captures the interaction of those three features, let us first motivate their desirability in the light of data-driven inference. While doing so, we will also introduce some terminology and notation.

To tackle the subtleties involved in  (i) we work with a language composed of two (finite) sets of propositional variables $\H = \{H_1,\dots, H_n \}$ and $\D = \{{d_1,\dots,d_l }\}$, standing for  \textit{hypotheses} and \textit{data}, respectively. We do not assume that the sets are disjoint, and sometimes we require them not to be, as in Proposition \ref{prop:rjaddrules} below. We denote by $\fm_{\H}$  and  $\fm_\D$ the set of propositional formulas generated by closing  $\H$  and $\D$, respectively, under the usual  connectives $\land,\lor,\neg$.  The set of all formulas is denoted by $ \fm $, i.e.  $Fm= \fm_{\D}\cup \fm_{\H}$. We denote elements of $\fm$ by lowercase Greek letters. 

The key element in our construction is  \textit{the degree to which data $\delta$ $\in \fm_{\D}$ rejects a hypothesis $\phi \in \fm_{\H}$}.   As will be clear in Section \ref{sec:prj} and in Section \ref{sec:rrj}, detailing the specific properties of degrees of rejection will lead to distinct consequence relations. This is why we refer to the Definition \ref{def:degrej} below as a \textit{blueprint} consequence relation.

To motivate (ii) recall that one key feature in the complicated process that leads to data-driven scientific knowledge  is the ``replicability'' of the tests or experiments which engender it. The relevance of this for the intended semantics of the consequence relations we are about to construct is that the same hypothesis may be confronted multiple times by the same data. A textbook situation in which this happens is when sampling with replacement from an urn. But of course, concrete cases of scientific inference will not be so easy to describe (mathematically). Moreover, scientists may make mistakes in the experimental setup, in the data analysis, or at any other point of the procedure. We accommodate those heterogeneous cases centring Definition \ref{def:degrej} on  the degree of rejection that a \textit{multiset} of data $\Delta$ from $\fm_{\D}$ puts forward against hypothesis $\phi\in \fm_{\H}$. 

We denote by $\multid$ the set of multisets built over the formulas in $\fm_\D$. Henceforth, abusing notation we will use curly brackets for denoting both a multiset and a set of elements\footnote{Recall that the set $\mathcal{M}_S$ of multisets built over a set $S$ can be more formally identified with the set $\mathbb{N}^S$ of functions from $S$ to $\mathbb{N}$, where $f\in \mathbb{N}^S$ stands for the multiset where each $s\in S$ occurs $f(s) = n$ times.}.  

 Finally, to motivate (iii), note that in scientific inference data is not \textit{certain}. This is one reason why the replicability of experimental setups is of critical importance. However, since we are putting forward a logical framework, we must provide our ideal and logically omniscient scientists with knowledge about the problem at hand which is \textit{not} revisable as a consequence of the experiment. To capture this we will assume that data-driven inference is performed under the assumption that the hypotheses in $\H = \{H_1,\dots, H_n \}$ are mutually exclusive and exhaustive. We capture this through the {following} classical formula $T$
 
 \[ {T : = \bigvee_{H \in \H} H \ \land \  
\bigwedge_{H,H' \in \H, \ H \neq H' }H \ra \neg H' .}\]

We are now ready to provide the blueprint definition of the degree $r$ to which data reject a hypothesis. Section \ref{sec:prj} and \ref{sec:rrj} will investigate two special cases which arise by specifying distinct ways of computing $r$.
 %, as well the information that the hypotheses whose degree of rejection is $\infty$ are false. 

\begin{definition}\label{def:degrej}
%\paolo{Let $\H\subseteq \fm_\H$}. 
We say that $r\colon \multid \times \H \to  \mathbb{R}_+ \cup \{\infty\}$ is a \emph{degree of rejection} if, for any $\c,\d\in\fm_\D$, $\Delta\in \multid$ and $H\in\H$, the following hold:
\begin{enumerate}
\item If $T,\c\models \d$,  then $r_\c(H)\geq r_\d(H)$.  
%\item %\paolo{If $\d \in \fm_H$, then} 
 \item %$\displaystyle{
     If  $T, H \models \d$ then $r_{\delta}(H)= 0$.
     \item  If $T, \delta \models \neg H$, then $r_\d(H)>0$ .
%   \footnote{\esther{Quindi se $T, H_1 \models \neg \d$ e $T, H_2 \models \neg \d$, potrei avere $0<r_{\delta}(H_1)<r_{\delta}(H_2)<\infty$ con rejection degrees differenti.}}
%\item If $H\models\neg \d$ then $r_{\delta}(H)\geq r_\delta(H')$ for each $H'\in\H$
   %\end{cases}} 
    %,\text{and}$ %1       &\, \text{ if }  H \models\neg \d, \\
 %   \item[] 
  \item If $r_\Delta {(H)}\not= \infty$, then $r_\Delta(H)=
  \sum_{\delta\in\Delta} r_\delta(H) $.
  \end{enumerate}
  \end{definition}
  When no confusion is likely to arise, we  simply speak of $r_\Delta$ as \lq\lq the degree of rejection" and denote it by $r$. 

Proposition \ref{prop:rl} below shows that the definition is well-posed. In fact the remainder of this paper provides examples of functions which satisfy Definition \ref{def:degrej} and explores how they provide the semantics for distinct consequence relations.

Part (1) of the definition states that, modulo $T$, logically stronger data provide a higher degree of rejection. If we do not insist on limiting our inferences to satisfiable data, this implies that observing contradictions provides the maximal degree of rejection for any hypothesis in $\fm_{\H}$. This is at odds with the scientific practice where we may observe a contradiction between the data and a hypothesis, but we typically do not observe contradictory data. As a consequence, we will tacitly assume that the elements of $\fm_{\D}$ are satisfiable even if the formalism does not require us to do so. 

Conditions (2) and (3) link degrees of rejection to classical logic. The former states that no positive degree of rejection can be provided when one observes a logical consequence, modulo $T$, of a given hypothesis. (3) states that a degree of rejection cannot be zero whenever the data, modulo $T$, entail the negation of the hypothesis. This is where Definition \ref{def:degrej} marks a genuine departure from the traditional idea that the relation between observations and hypotheses is governed entirely by classical deduction, and in particular through \textit{modus tollens}. Whilst it may happen, in special cases, that   $r_\d(H)=1$ for data and hypotheses such that  $\delta \models \neg H$, this need not be true in general and $r$ can take any positive value up to $\infty$. This departure from classical logic is motivated, as noted in the introductory section of this work, by the fact that data obtained in any experimental setup may translate into incomplete, scarce, noisy, or otherwise imperfect evidence. Condition (3)  ensures that the  semantics of data-driven inference accommodates this central feature of scientific inference at root.

Finally, condition (4) states that the degree to which a hypothesis is rejected by data is computed as the sum of the degrees of rejection obtained trough each single piece of data in the multiset  $\Delta$. As a justification for aggregating with the sum, as opposed to any other function, at present we can offer its simplicity and the interesting consequences it leads to. Note however that additive aggregations are used across many (distinct) formalisms in uncertain reasoning, including  the sums of ``uncertainties'' in Adams's probabilistic logic \cite{Adams1998}, and the sums of ``risks'' in Giles's game-theoretic foundation of many-valued logic \cite{Giles1982}.

The next step towards the intended semantics of our blueprint consequence relation revolves around the set 
$\maxh_\Delta$ of hypotheses in $\H$ which are \emph{least rejected} by the data in $\Delta$:  
\begin{equation}\label{eq:argmin}
\maxh_{\Delta} = \argmin_{H\in \fin{\Delta}} r_\Delta(H)  .
\end{equation}
where $\fin{\Delta}$ is the set of hypotheses with a finite degree of rejection. 
%\mathbb{R}_+finite. 
The rationale behind least-rejection is that the elements of $\maxh_{\Delta}$ are  those hypotheses which ``did best against data'' in $\Delta$.

\begin{remark}\label{rem:empty}
	 $\maxh_{\Delta}$ is empty only when there is no hypothesis with a finite rejection degree, i.e. if  {$r_\Delta(H) =\infty$} for each $H\in \H$.%, for some $\delta\in \Delta$.
\end{remark}

We are now ready to define the blueprint consequence $\rj$ as classical model-preservation under least-rejection. 
%REPLY: we removed the verbal explanation, which we realised would not be useful here.

%In words, 
%\paolo{ data entail a conclusion under the $\rj$  consequence, if and only if the formula in the conclusion holds under all the hypothesis that are least rejected by the premises expressing the data.}
%$\rj$ takes data as premisses and returns, for conclusion, a formula expressing the hypotheses which have been least rejected by the data.

\begin{definition}[RJ-consequence]\label{def:Rj}
Let $\Delta$ be a multiset of formulas in $\fm_\D$ and $\f\in \fm_\H$.  We say that $\f$ is an \emph{RJ-consequence of $\Delta$}, written $\Delta\rj\f$, if $T, H\models \f$, for each $H\in \maxh_\Delta$.
\end{definition}

The instantiations of the blueprint consequence to be investigated in this paper will make explicit the nature of the ordering among rejected hypothesis which makes the minimality featuring in \eqref{eq:argmin} precisely defined. But readers who are familiar with the preferential variety of non-monotonic logic will realise immediately from Definition \ref{def:degrej} that our blueprint belongs to that family.

As an immediate consequence of Definition \ref{def:Rj} and Remark \ref{rem:empty} we have that $\rj$ is explosive in the sense that if there are no least-rejected hypotheses, then all hypotheses are rejected by data.

\begin{lemma}\label{prop:rj-explosive}
	Suppose $\maxh_{\Delta} = \emptyset$ . Then $\Delta\rj\f$ for any {$\f\in\fm_\H$}.
\end{lemma}

%Our first results compare  $\rj$ with \textit{cumulative consequence relations}, aka System C. Specifically, 
Proposition \ref{lemma:rj-soundness} shows that $\rj$ satisfies  (multiset versions of) the rules (RWE), (LLE), and (AND), whereas  Proposition \ref{genmon} shows that $\rj$ fails, desirably, unconstrained monotonicity. Finally, Proposition \ref{prop:rjaddrules} shows that $\rj$ satisfies the rules of  rational consequence relation with the important exception of (OR), which will be  discussed later.

\begin{table}[h!]
\centering
\begin{tabular}{ccc}
%\infer[\text{(REF)}]{\f\rj\f}{} & 
\infer[\text{(RWE)}]{\Delta \nm \p} {\Delta \nm \f & &\f\models\p} & \\[2.0em]
    \infer[\text{(LLE)}]{\Delta,\d\nm\f}{\Delta,\c\nm \f & &\models\c\leftrightarrow\d }  & 
     \infer[\text{(AND)}]{\Delta \nm   \f \land \p } {\Delta\nm  \f & & \Delta \nm  \p}   \\[2.0em] 
% \infer[\text{(AND)}]{\Delta \rj   \f \land \p } {\Delta\rj  \f & & \Delta \rj  \p}   \\[2.0em]
 
%\infer[(cMon)^*]{\Delta,\f\rj_\a \p}{ \Delta_1\rj_{\a_1} \p & \Delta_2 \rj_{\a_2} \xi } & 
%\infer[(cMon)^*]{\Delta,\f\rj_\a \p}{ \Delta\rj_\a \p & T_\Delta, \f\not\models\neg \p } & 
%\infer{excl}{H_i \rj \neg H_j }{}
%\infer[{(RW)^*}]{\f\rj_\a\xi} {\Delta \rj_\a \p & \p,T_\Delta \models\xi}
% \\[2.0em]
%\infer[(cMon)^g]{\Delta,\f\rj_{\a_1\odot\a_2} \p}{\Delta\rj_{\a_1}\f & \Delta\rj_{\a_2} \p  } 
%\infer[(\land l )]{\Delta,\f\land\xi\rj \p } {\Delta,\f,\xi \rj\p & &\neg\f,\neg \xi \models \bot}
%\infer[(\lor)]{\Delta,\f\lor\xi \rj    \p } {\Delta, \f\rj  \p & \Delta,\xi \rj  \p}  &
%\infer[(\neg)] {\Delta \rj_{1-\alpha}  \neg H_i } {\Delta\rj_{\alpha} H_i}  \\[2.0em]
%\infer[(\subseteq)_{I\subseteq I'}]{\Delta\rj_{I'}\f}{\Delta\urj \f} & %\infer[mon]{\Delta}{}
%\\[2.0em]& T_\Delta, \f\not\models\neg \p 
\end{tabular}
\caption{System $RJ$}\label{tab:mr}
\end{table}

%
%\begin{table}[h!]
%\centering
%\begin{tabular}{ccc}
%\infer[\text{(REF)}]{\f\q\f}{}  &\infer[\text{(RWE)}]{\f\q\xi} {\f \q \p & \p\models\xi} \\[2.0em]
%    \infer[\text{(LLE)}]{\p\q\xi}{\f\q \xi & \f\models\p &\p\models\f} & 
%                                                                         \infer[\text{(CMO)}] {\f\land\xi \q \p}{\f\q \p & \f\q\xi } \\[2.0em]
%
%  \infer[ 	\text{(AND)}]{\f\q \p\land\xi} {\f\q \p & \f\q\xi}  & %\infer[(OR)]{\f\lor                                                               \p\q \xi}                                                               {\f\q \xi                                            & \p \q\xi}
%  \\[2.0em]
%%\multicolumn{2}{c}{\infer[\text{(rMon)}]{\f\land\xi\nm\p}{\f\nm\p & \f\not\nm\neg\p}}\\[2.0em]
%\end{tabular}
%\caption{System  $Rj$}\label{tab:mr}
%\end{table}
 \begin{restatable}{proposition}{Firstprop}\label{lemma:rj-soundness}
RJ-consequence satisfies the rules in Table \ref{tab:mr}.
\end{restatable}

See the Appendix for the proof. 

%, with $\tau_\Delta(H_i)\in I$. Hence, $T_\Delta, H\models\f\land\p$ for any $H\in \maxh_{\Delta}$, that is, $\Delta\rj\f\land\p$.

% and let $H \in \maxh_{\Delta}$. First, note that, since $\f$ occurs both on the right and the left hand side of $\rj$, we can only have $\f\in\fm_\H$

%By the premise $\Delta\rj\f$, we further have $H\vdash \f$. But this means that $\f\not\vdash\neg H$. Hence $r_\f(H)=0$ and thus $r_{\Delta,\f}(H) = r_\Delta$. But this means that $H\in \maxh_{\Delta,\f}$, hence by the premise $\Delta,\f\rj\p$, we finally have $H\vdash\p$. Hence, we have shown that for any $H\in \maxh_{\Delta}$, we have $H\vdash \p$, that is, $\Delta\rj\p$.

It is well-known \cite{Makinson2005,Makinson2012,Hawthorne2007} that (non-monotonic) logical systems satisfying (AND) provide   \textit{qualitative} representations of reasoning under uncertainty. This may appear to be at odds with the ``numerical'' semantics provided by degrees of rejection. However this is not sufficient to make $\rj$ quantitative. In fact the magnitude of  degrees of rejection are immaterial in  \eqref{eq:argmin}, where only comparisons of such degrees are relevant.

Next, we ensure that $\rj$ does not satisfy  unconstrained monotonicity (MON).  Since our consequence relations are defined on multisets of premisses, two distinct formulations of unconstrained monotonicity must be taken into account: 
\[\infer[\text{(AMON)}]{\c\land\d\nm \f}{\c\nm\f} \qquad \infer[\text{(MMON),}]{\Delta,\Delta'\nm \f}{\Delta\nm\f}\] where, as usual, $\nm$ stands for an arbitrary consequence relation. 

 \begin{restatable}{proposition}{Secondprop}\label{genmon}
	Neither \emph{(AMON)} nor \emph{(MMON)} hold for $\rj$.
\end{restatable}

See the Appendix for the proof.

%cautious monotonicity
In preparation for the last proposition of this Section, which pins down the conditions under which reflexivity, rational monotonicity and cut are satisfied by $\rj$, recall that we work with distinct languages for data and hypotheses. Since Proposition \ref{lemma:rj-soundness} involves only rules which do not feature individual formulas on both sides of $\rj$, the proof carries through even if the two languages are disjoint. This possibility, however, must be ruled out for  $\rj$ to satisfy the rules of the next Proposition. Moreover we must set to infinity the rejection degree provided by contradictory data against any hypothesis. Recall also that we denote by  $T$ the formula expressing that hypotheses form a partition.

\begin{table}[h!]
	\centering
	\begin{tabular}{ccc}
\infer[\text{(REF)}]{\f\nm\f}{} & \infer[\text{(RMO)}]{\Delta,\f\nm \p}{ \Delta\nm \p & \Delta\not\nm\neg\f}
 \\[2.0em] \infer[\text{(CUT)}]{\Delta \nm \p}{ \Delta,\f\nm \p & \Delta\nm\f} &
\infer[\text{(CMO)}]{\Delta,\f\nm \p}{ \Delta\nm \p & \Delta\nm\f}
\end{tabular}
\caption{}\label{tab:rjaddrules}

\end{table}

 \begin{restatable}{proposition}{Thirdprop}\label{prop:rjaddrules}
Suppose $\fm_\D\cap\fm_\H \neq \emptyset$. Assume further that if $T\models\neg\delta$ then $r_{\delta}(H) = \infty$ for each $H\in\H$.
Then $\rj$ satisfies the rules in Table \ref{tab:rjaddrules}.
\end{restatable}

See the Appendix for the proof.

For $\rj$  to be a rational consequence relation it  should also satisfy (OR). It does not, as detailed in Section \ref{sec:invalid-rules-for-prj}, where we also explain why this is indeed desirable.

As a concluding remark on our blueprint consequence relation, note that the \textit{supraclassicality} of  $\rj$ is an immediate corollary of Proposition \ref{prop:rjaddrules} by concatenating the arguments for (REF) and (RWE).

\begin{corollary}
Let $\f,\p\in\fm_\D\cap\fm_H$. If $\f\models\p$ then $\f\rj\p$.	
\end{corollary}

\section{Data-driven probabilistic rejection}
\label{sec:prj}
This section investigates two ways of defining  degrees of rejection through data-driven probabilities. The first is  discussed in Section \ref{sec:likelihood} and captures the core idea of maximum likelihood inference. The second is discussed in  Section \ref{sec:p-value} and is closely related to significance inference and NHST. Each gives rise to a distinct consequence relation. We argue that both capture interesting aspects of statistical inference.

Recall that a \textit{probability function} on $\fm$ is a map
	\makebox{$P:\fm \rightarrow [0,1]$} which satisfies normalisation and additivity, i.e.
	\begin{enumerate}
		\item If $\models \phi$ then $P(\phi)=1$ and
		\item  If $\models \neg (\phi \wedge \theta)$ then $P(\phi\vee\phi)=P(\phi)+P(\theta)$.
	\end{enumerate}

Throughout this section we will work with  \textit{conditional} probability functions of $\fm$, i.e. $P:\fm \times \fm^+ \rightarrow [0,1]$ defined by
\begin{equation}\label{eq:conditional}
P(\d\mid H)=\frac{P(\d\wedge H)}{P(H)},
\end{equation}
where $\fm^+$ denotes the set of elements of $\fm$ whose probability is strictly positive. Formally, this restriction  guarantees we do not divide by 0. Conceptually, it prevents bizarre applications of statistical hypothesis testing where the hypotheses being tested are known to be false. 

Functions defined as in \eqref{eq:conditional} above are  normalised on tautologies and finitely additive. That is, they satisfy:
\begin{equation}\label{eq:additivity-conditional}
	\text{If } \models \neg (\d\wedge \c) \text{ then } P(\neg\d\lor \neg \c\mid H) = P(\d\mid H) + P(\c \mid H).
	\end{equation}
%P(\neg\d\mid H) + P(\neg\c \mid H).
To express the condition to the effect that {$\delta$ and $\gamma$} 
%$\delta_1$ and $\delta_2$ 
are incompatible, we will also write $\d,\c \models \bot$, where $\bot$ is any classical contradiction. Finally  the following holds, and will be used extensively below:
\begin{equation}\label{eq:conditional-negation}
	P(\neg \delta \mid H)= 1-P(\delta \mid H). 
\end{equation}

\subsection{\textit{l}RJ-consequence}\label{sec:likelihood}
To illustrate the idea of  likelihood-based rejection, consider $\H = \{H_1,H_2 \}$ and let \begin{equation}\label{eq:s}
	r^l_\delta(H_i) = P(\neg  \delta|H_i)
\end{equation} for  data $\delta\in \fm_{\D}$ and $i\in\{1,2\}$. In this situation $H_1$  belongs to the set of least rejected hypotheses $\maxh_\delta$ if and only if  $r_\delta(H_1) \leq r_\delta(H_2)$. By \eqref{eq:conditional-negation} then  $H_1 \in \maxh_\delta$ if and only if $1 - P(\delta|H_1) \leq 1- P(\delta|H_2)$, i.e. 
	\begin{equation*}
		H_1 \in \maxh_\delta 	\text{ if and only if } P(\delta| H_1) \geq P( \delta|H_2).
	\end{equation*}
So, when $r$ is computed as in \eqref{eq:s}, the set of least rejected hypotheses under $\delta$  pins down the set of hypotheses which maximise  likelihood. Hence we call $r^{l}$ the \textit{maximal likelihood rejection function}.

As above we assume the additive aggregation of the degrees of rejection over multisets of data:
$$r^l_\Delta(H ) = \sum_{\delta\in\Delta} r^l_\delta(H).$$

\begin{proposition}\label{prop:rl}
	$r^{l}_{\delta}(H)$ is a degree of rejection.
\end{proposition}
\begin{proof}
We must show that $r^{l}$  satisfies parts (1-3) of Definition \ref{def:degrej}. For part (1), recall that  $\gamma\models\delta$ entails $P(\gamma|H)\leq P(\delta|H)$ for all conditional probability functions defined as in \eqref{eq:conditional}. By \eqref{eq:conditional-negation}, {we get $r^l_\delta(H) = P(\neg \delta|H) \leq P(\neg \gamma|H)=r^l_\gamma(H)$, as required.} As for (2) observe that  if $H\models \delta$ then $P(\delta|H) = 1$. By \eqref{eq:conditional-negation} we have $r^l_\delta(H)= P(\neg \delta|H) = 1- P(\delta|H) = 0$, as required. A similar argument delivers (3). For  if $H\models\neg \delta$, we have $r^l_\delta(H)= P(\neg \delta|H) = 1- P(\delta|H) = 1 $, which is greater than 0.
\end{proof}

We are now ready to define $\prj$, which  arises from  $\rj$ by imposing $r^{l}$ as the  degree of rejection. Recall that  $\Delta\rj\phi$ if $\f$ is a classical consequence of $H$ for each minimally rejected hypothesis $H\in \maxh_\Delta$. 

\newpage

\begin{definition}\label{def:prj}
  We say that \emph{$\phi$ is a \prjt-consequence of $\Delta$}, written \mbox{$\Delta\prj\phi$}, if
  \begin{itemize}
  \item[i)]$\Delta\rj\phi$, and
   \item[ii)] $r$ is computed according to \eqref{eq:s}, i.e. $r=\rp.$
  \end{itemize}
\end{definition}

%By inspecting the Definition,  

%Our first result shows that the inferential counterpart of this is that conjunction on the left, i.e.holds for $\prj$. 
 As a consequence of Proposition \ref{prop:rl}, $\prj\supset \rj$. The key difference between $\prj$ and the blueprint $\rj$ lies in the fact that the former satisfies both \textit{conjunction on the left} and its \emph{converse}, namely: 
\[ \infer[\text{(AND)}_l]{\Delta,\d\land\c\nm \f} {\Delta,\d,\c \nm\f & \neg\c,\neg \d \models \bot }\]
%& }
and 
\[ \infer[\text{(AND)}_l^{con}]{\Delta,\d,\c\nm \f} {\Delta,\d\land \c \nm\f & \neg\c,\neg \d \models \bot} \]

while this is not the case in general for $\rj$, as we will show in the next section.

Before proving that $\text{(AND)}_l$ holds for $\prj$, let us argue in favour of its desirability. By inspecting the likelihood-based rejection function $r^l$ \eqref{eq:s} it is apparent that the semantics of $\prj$ builds on ``the probability of negations''. It is therefore the additivity  of conditional probability functions \eqref{eq:additivity-conditional}  which  delivers the validity of this rule.

\begin{proposition}\label{lemma:prj-soundness}
In addition to all the rules in Table \ref{tab:mr} \prjt-consequence relations satisfy also $\emph{(AND)}_l$ and $(\text{AND})_l^{con}$.
\end{proposition}

\begin{proof}
%	For the first, in the light of Proposition \ref{prop:andandconv}, we only need to show that $r^l_{\c\land\d}(H)\leq r^l_{\c}(H) + r^l_{\d}(H)$ for any $H\in \H$. 
%This follows immediately from the additivity of conditional probability, since:
%\[r^l_{\c\land\d}(H) = P(\neg\c\lor \neg \d|H) = P(\neg \c|H) + P(\neg\d|H) - P(\neg \c \land \neg \d |H) \leq r^l_{\c}(H) + r^l_{\d} (H)\] 
%
%For the rule $\text{(AND)}_l^{con}$, it suffices to recall that, in case $\neg\c,\neg\d \models \bot$, we get $P(\neg \c \land \neg \d |H)=0$. Hence by the above equation we will obtain 
%$r^l_{\c\land\d}(H) = r^l_{\c}(H) + r^l_{\c} (H)$. In virtue of Proposition \ref{prop:andandconv} this entails that $\text{(AND)}_l^{con}$ holds. 

Suppose $\Delta,\d,\c \nm\f$ and $\neg\d,\neg \c \models \bot$. 
By \eqref{eq:s}, we have $r^l_{\d\land\c}(H) = P(\neg(\d \wedge \c) \mid H)$. The right hand side, by the elementary properties of conditional probability functions equals $P(\neg \d \lor \neg \c \mid H)$. By the second hypothesis of (AND)$_l$  we know that $\neg\d$ and $\neg \c$ are incompatible. Hence, by additivity  $r^l_{\d\land\c}(H) =P(\neg \d\mid H) +  P( \neg \c \mid H)$ which equals $r^l_{\d}(H) + r^l_{\c}(H)$. This entails that for each $H$, we have $r^l_{\Delta,\c,\d}(H) = r^l_{\Delta,\c\land\d}(H)$, hence $\maxh_{\Delta,\c,\d} = \maxh_{\Delta,\c\land\d} $ and thus $\Delta,\c,\d\rj \f$  if and only if $\Delta,\c\land\d\rj \f$. This finally shows that
 $r^l$ satisfies both $(\text{AND})_l$ and $(\text{AND})_l^{con}$.
\end{proof}
%

%Inspection of the proof shows that also the converse $\text{(AND)}_l^{con}$ of $\text{(AND)}_l$ holds for $\prj$:
%\[ \infer[\text{(AND)}_l^{con}]{ \Delta,\d,\c\nm \f} {\Delta,\d\land\c\nm \f &\neg\d,\neg \c \models \bot}. \]

\subsection{(OR) is not valid for $\prj$}\label{sec:invalid-rules-for-prj}
As a consequence of Proposition \ref{prop:rl}, if a rule is not valid for $\prj$, then it is not valid for $\rj$ either. And, as anticipated, a noticeable failure for the blueprint consequence relation is the rule known in the non-monotonic logic literature as \textit{disjunction in the premisses} (OR).  Hence both  $\prj$ and its predecessor $\rj$, fall short of being preferential consequence relations in the sense recalled in Section \ref{sec:nm}. 

Before showing this, let us make explicit that this \emph{is} in fact desirable in light of the intended semantics of the consequence relation(s). Recall that $\prj$ is  likelihood based: the degree to which an observation $\delta$ rejects a hypothesis $H$ is computed as $P(\neg \d \mid H)$. If we were to insist on satisfying (OR) we would need to guarantee that $P(\neg (\delta_i \lor\delta_j)\mid H) = P(\neg \delta_i\land \neg \delta_j\mid H)$  can always be computed in terms of $P(\neg \delta_i \mid H)$ and $P(\neg \delta_j\mid H)$. But it is well known that is cannot be done in general, for (conditional) probability functions are not compositional with respect to conjunction -- see Section 1.1 of \cite{Hosni2023a} for a general discussion.  %\cite{Hawthorne2007,Makinson2012}. 

We show that $\text{(OR)}$ fails for \prjt-consequence by arguing that the  weaker rule (XOR), introduced in   \cite{Hawthorne2007}, fails too:
\[\infer[(\text{XOR})] {\Sigma,\d\lor \c \nm \f.}{\Sigma,\d\nm \f & \Sigma,\c\nm \f & \d,\c\models \bot }  \]

%\e{\[\infer[(\text{XOR})] {\Sigma,\f\lor \p \nm \xi.}{\Sigma,\f\nm \xi & \Sigma,\p\nm \xi & \f,\p\models \bot }  \]}

%\[\infer[(\text{OR})] {\Delta,\psi \lor \c \nm \f.}{\Delta,\psi\nm \f & \Delta,\c\nm \f}  \]

\begin{lemma}\label{failsxor}
\emph{(XOR)} does not hold for $\prj$.
\end{lemma}
\begin{proof}
To construct the required counterexample, let $\H = \{H_1,H_2,H_3\}$ and suppose  $\c,\d\models \bot$, i.e. $\c$ and $\d$ are logically incompatible data. Consider the following probability assignments: 
\begin{align*}
P(\neg \d\land\c \mid H_1) = 0.5  & \quad P(\neg \d\land\c \mid H_2 ) = 0.3 &   P(\neg \d\land\c \mid H_3 ) = 0.7 \\
P(\d\land\neg \c \mid H_1 ) = 0.4 & \quad  P(\d\land\neg \c \mid H_2) = 0.5 &  P( \d\land\neg\c \mid H_3 ) = 0.1\\
 P(\neg\d\land\neg\c \mid H_1 ) = 0.1 & \quad P(\neg\d\land\neg\c \mid H_2) = 0.2  &  P(\neg \d\land\neg\c \mid H_3 ) = 0.2.
\end{align*}

Since $\c,\d\models \bot$ we have 
\[P(\d\land \c \mid H_1 ) = P(\d\land\c \mid H_2) = P(\d\land\c\mid H_3) =0. \] 

Now, by simple computations, we obtain 
\begin{align*}
	r^l_{\delta}(H_1) &= P(\neg\delta \mid H_1) = 0.6, \\
	r^l_{\delta}(H_2) &= P(\neg\delta \mid H_2) = 0.5, \\
	r^l_{\delta}(H_3) &= P(\neg \delta \mid H_3) = 0.9,
\end{align*}
hence $\maxh_{{\d}} = \{H_2\}$ and \[{{\d}\nm H_2\lor H_3}.\]
Moreover,
\begin{align*}
	r^l_{\c}(H_1) &= P(\neg{\c}\mid H_1) = 0.5,	\\ r^l_{\c}(H_2) &=P(\neg{\c}\mid H_2) = 0.7,\\ r^l_{\c}(H_3) &=P(\neg {\c}\mid H_3)= 0.3,
\end{align*}
hence $\maxh_{{\c}} = \{H_3\}$ and \[{{\c}\nm H_2\lor H_3}.\]
We have thus shown that the premises of XOR are satisfied. 
On the other hand $r^l_{{\d}\lor{\c}}(H_1)= P(\neg {\d}\land \neg {\c}\mid H_1)= 0.1$, 
$r^l_{\d\lor\c}(H_2)= P(\neg \d\land \neg \c\mid H_2)= 0.2$, and $r^l_{{\d}\lor{\c}}(H_3)= P(\neg {\d}\land \neg {\c}\mid H_3)= 0.2$, so that $\maxh_{{\d}\lor {\c}} = \{H_1\}$, and finally \[{{\d}\lor{\c}\not\nm H_2\lor H_3}\]
that is, the conclusion of the rule (XOR) is not satisfied.
\end{proof}

\subsection{$\alpha$-RJ consequence}\label{sec:p-value}

Recall that likelihood is the probability of \textit{observed} data conditional on a (statistical null) hypothesis. It is a calculated, rather than estimated, quantity.  A long-standing controversy in the methodology of statistics revolves around whether likelihood is all that one should take into account when making data-driven inferences, see \cite{Royall1997} for an articulated account.

As recalled in the introductory Section, the practice known as Null Hypothesis Significance Testing (NHST) combines several aspects of statistical methodology and gives a prominent role to the p-value which is the calculated  probability of obtaining \textit{hypothetical} data, namely data at least as improbable as those observed  conditional on the null hypothesis. Clearly the p-value is not entirely data-driven, as it features data which in fact one has \textit{not} seen -- see \cite{Wasserstein2016} for an extensive discussion on this and related point. 

For our present purposes it is enough that we capture the qualitative comparison of the conditional probabilities involved, and in particular the fact that when computing p-values, we consider the cumulative probability function defined over all events which are less probable than those observed conditional on $H_0$.

To capture this new feature, we add to the language of $\rj$ and $\prj$ an operator $t_H$ which associates to each hypothesis $H$, the formula $t_H(\delta)$. The operator is  characterised as follows:
\begin{equation}\label{eq:t}
	\delta \models t_H(\delta), \qquad  \infer{t_H(\c)\models t_H(\d)}{\c\models \d}, \qquad t_H(t_H(\d))\models t_H(\d).
\end{equation}

From the logical point of view, $t_H$ is just  a closure operator over the set of models of the formula $\delta\in\fm_{\D}$. The rules which characterise it express its extensiveness, monotonicity and idempotence, respectively.
 
To illustrate, $\d\models t_H(\d)$ states that the set of models of  $\d$, which translate probabilistically as the event that $\delta$ occurs, is a subset of the set of models of $t_H(\d)$, the event that data at least as improbable than $\delta$ occur. The other rules  read similarly.

The consequence relation to be introduced in Definition \ref{def:arj} below departs in another way from $\prj$. Recall that NHST combines ideas by Fisher about significance inference with ideas by Neyman and Pearson about statistical testing, albeit against the will of all parties involved. Whilst the Fisher approach insists on p-values providing a degree of inconsistency between the data and the null hypothesis, the Neyman-Pearson one is oriented towards the (binary) decision whether to reject or not the null hypothesis. To do so, a threshold $\alpha$ is fixed at design time which  identifies the so-called \textit{rejection region}: the null hypothesis is rejected if the observations conditional on the null hypothesis  falls within this region.  The key feature of the consequence relation $\arj$ is that of combining the qualitative nature of the binary decision with the graded interpretation of the p-value. 

Two preliminary remarks before we get to the definitions. First,  we represent the p-value associated to hypothesis $H$ by the quantity $P( t_H(\delta) \mid  H)$. Building on this,  the rejection degree for a hypothesis $H\in \H$ is set to 1 minus the p-value, \textit{provided} the latter is below a fixed threshold $\a$.  Second, in Definition \ref{def:degarj} below, and in the reminder of this Section, we will assume that hypotheses in $\mathcal{H}=\{H_1,\ldots, H_n\}$ are mutually exclusive and jointly exhaustive. 
%Finally, and without loss of generality, we will assume that {$H_n$} is classically equivalent to the negation of all the other hypotheses.  

 %Indeed, for reasons that will be apparent, we will be interested in taking ``less extreme events''. }
%\newpage
\begin{definition}[$r^\alpha$]\label{def:degarj}
%Fix $\a\in(0,1)$ and suppose $\delta\in\fm_\D$. 
%Then	 \[  
%r^\a_{\delta} (H_n )= \begin{cases}
%	1       &\, \text{ if } T, H_n \models\neg \d, \\
%	0  & \, \text{ otherwise. }
%\end{cases} 
%\] 
%If $H_i\not = H_n$ then 
Fix $\a_i\in[0,1)$ for each $H_i\in \H$ and suppose $\delta\in\fm_\D$. Then
\begin{equation*}
   r^{\a}_{\delta}(H_i) = 
  \begin{cases}
  1-P( t_{H_i}(\delta) \mid  T\land H_i)  \qquad  \text{ if } P( t_{H_i}(\delta) \mid  T\land H_i)\leq \a_i  \\
  0  \qquad\qquad\qquad\qquad\qquad \quad \; \text{ otherwise.}
\end{cases}
\end{equation*}

Finally,
\[r^\a_{\Delta}(H) = \sum_{\delta\in\Delta} r^\a_\delta(H),\]
for any $H\in \H, \Delta\in \multid$.
 \end{definition}

%  
%Note the asymmetry between the rejection of $H_n$ and that of other hypotheses. Since $H_n$ plays the role of a \emph{catchall} hypothesis it is not reasonable to assume that it can be rejected by data because it is 
%{not} directly under scrutiny. So if one rejects $H_n$ it can only be on  logical grounds.

%\paolo{This use of a catchall hypothesis that is not under test matches scientific practice, which often departs from the ideal condition of having an exhaustive set of hypotheses, all explicitly stated and testable.} 
%\footnote{\paolo{Ho provato a riformulare rendendo chiaro che può essere una situazione non ideale, ma siamo comunque capaci di catturarla. Non so se va bene...non metterei tanto la rejection logica, perché  lo vedo come un caso "degenere" che non riflette direttamente il punto qui. Altrimenti per me andrebbe bene  togliere}}

%\comment{This matches  the practice of NHST, when testing for rejection \emph{only} a null hypothesis, without explicit consideration of the alternative(s).}{Vero ma non lo direi così perché è spesso considerato un difetto metodologico. Sarei per portare un esempio di rejection logica, oppure lasciare perdere e tenerci questa risposta nel caso in cui ci viene chiesto chiarimento}

\begin{lemma}\label{genincl}
$r^\alpha$ is a degree of rejection.
\end{lemma}
  
\begin{proof}
We need to check that the rejection degree of Definition	\ref{def:degarj} satisfies Definition \ref{def:degrej}.

For condition (1), in suppose $\c\models \d$, for $\c,\d\in\fm_\D$. Since $\c\models\d$,  we have that $t_{H_i}(\c)\models t_{H_i}(\d)$, hence by the monotonicity of probability functions, $P( t_{H_i}(\c) \mid H_i ) \leq   P( t_{H_i}(\d) \mid H)$. From this it follows immediately that 
if $P( t_{H_i}(\c) \mid H_i )\geq \a_i$, then $P( t_{H_i}(\d) \mid H)\geq \a_i$, and $1- P( t_{H_i}(\c) \mid H_i )\geq 1 - P( t_H(\d) \mid H_i)$, hence $r^\a_\c(H) \geq r^\a_\d(H)$, as required. 

The remaining conditions are established similarly, hence we omit checking them explicitly. 
% 
%As to condition (2) it is immediate to see that if $H\models \c$
%, then $P(t_H(\d)|H) = 1$, hence $r^\a_\d(H)=0$ for any $\a\in (0,1)$. \paolo{Similarly, for condition (3) if $H\models \neg d$, then $P( t(\d)|H) = 0 < \a$, hence $r^\a_\d(H)=1>0$.  
%Finally, condition (4) is obvious by definition.}
\end{proof}
We now define  $\arj$ as follows. 
{
\begin{definition}[$\arj$]\label{def:arj}
  We say that \emph{$\phi$ is  \arjt-consequence of $\Delta$}, written \mbox{$\Delta\arj\phi$}, if
  \begin{itemize}
  \item[i)]$\Delta\rj\phi$, and
   \item[ii)] $r$ is computed according to Definition \ref{def:degarj}, i.e. $r= r^\alpha.$
  \end{itemize}
\end{definition}
}

The next example shows that  $\arj$ formalises the key inferential step in a stylised example of a NHST procedure.

\begin{example}
Consider the null hypothesis $H_0$ and recall that we are under the blanket assumption to the effect that hypotheses form a partition. Hence we can write $\H = \{H_0, \neg H_0 \}$ instead of $\H= \{H_0,H_1\}$.
We fix $\a_0\in(0,1)$  and let $\a_1 = 0$. For simplicity, we drop the index of the former, denoting $\a_0$ by $\a$.
Definition \ref{def:degarj} gives us $r^\a_\delta(H_0) =1-P(t(\delta)|H_0) $ whenever $ P(t(\delta)|H_0)\leq \a$, and $r^\a_\delta(H_0)=0$ otherwise. On the other hand $r_{\delta}(\neg H_0) = 1$ only if $P(t_{\neg H_0}(\delta)|\neg H_0) = 0$, and  $r_{\delta}(\neg H_0) = 0$ otherwise. This accounts for the fact that in NHST one does not consider $\neg H_0$ to be under scrutiny for possible rejection. Thus $\neg H_0$ is rejected only in case of data practically impossible under $\neg H_0$ (e.g. in case of logical contradictions), while  the focus is wholly on whether the null hypothesis $H_0$ is rejected.
%And for $\neg H_0$,  we have
%$$r^\a_{\delta} (\neg H_0 ) = \begin{cases}
%	1       &\, \text{ if } \neg H_0 \models\neg \d, \\
%	0  & \, \text{ otherwise. }
%\end{cases}	
%	$$
Hence, whenever $P(t_{\neg H_0}(\delta)|\neg H_0) \neq 0$, we have that $\maxh_\delta = \{ \neg H_0\}$ if and only if $r^\a_{\delta}(H_0)\neq 0$, i.e. $P(t(\delta)|H_0) \leq \alpha$. This means in turn that
$$ \delta \arj \neg H_0 \text{ if and only if } P(t(\delta)|H_0) \leq \alpha,$$
i.e. \textit{the null hypothesis is rejected exactly when the p-value is below the fixed significance level }(equivalently falls within the \textit{rejection region}).
 \end{example}

Note that the above example also illustrates the asymmetry between \emph{rejecting} the null hypothesis, i.e. $\delta\arj \neg H_0 $  and \textit{retaining} it. Indeed, if $P(t(\delta)|H_0) >\alpha$, we  have $r^\a_\delta (H_0) = r^\a_\delta(\neg H_0) = 0$, hence $\maxh = \{H_0,\neg H_0\}$. So by Definition \ref{def:arj} we have both $\delta\not\arj  H_0$  and  $ \delta\not\arj  \neg H_0.$

% Now $H_1$ belongs to the set of least rejected hypotheses $\maxh_\delta$ if  $r_\delta(H_1) \leq r_\delta(H_0)$, i.e. if
%$1 - P(t(\delta)|H_1) \leq 1- P(t(\delta)|H_0)$, hence 
%\begin{equation}\label{eq:lr}
%    P(t(\delta)| H_1) \geq P(t(\delta)|H_0).
%  \end{equation}
%

The last proposition in this Section compares  $\arj$ and   $\prj$, by showing that neither  $\text{(AND)}_l$ nor $\text{(AND)}_l^{con}$, are satisfied by $\arj$ (recall that by Proposition \ref{lemma:prj-soundness} they are both satisfied by $\prj$).

%Recall that the semantics of $\arj$ assigns a special role to $H_n$, the catchall hypothesis. This difference is responsible for the fact that there is a rule which is satisfied by $\arj$ and not by $\prj$. 
%\begin{proposition}\label{negn}
%Let $\H = \{H_1,\dots,H_n \}$. The consequence relation $\arj$ satisfies the rule:
%\[\infer[(NEG_n)]{\delta \not \nm  \neg H_n}{ \delta\not\models\neg H_n},\]
%but $\prj$ does not.
%\end{proposition}
%\begin{proof}
%For the first part, suppose the premise of $(NEG_n)$ holds. Then we have  $r^\a_\delta(H_n)=0$, by Definition \ref{def:arj}.  But this implies that $H_n\in \maxh_{\delta}$. Then, since clearly $H_n\not \models \neg H_n$ we get $\d\not\nm \neg H_n$, as required.
%
%To see that  $(NEG_n)$  does not hold for $\prj$  take $r=r^l$, $\H = \{H_1,H_2\}$ with $ H_1\models \neg H_2$, and  $\d\in\fm_\D$ such that $\d\not\models\neg H_2$ and $r_\d(H_1)\leq r_\d(H_2)$. We thus  have $\maxh_\d = \{H_1\}$ and $H_1\models \neg H_2$, so that $\d\nm \neg H_2$, as required.
%\end{proof}
%
%Whilst the above may suggest that there is an inclusion relation between $\arj$ and $\prj$, this is not the case since while
 \begin{restatable}{proposition}{Fourthprop} Both $(\text{AND})_l$ and $(\text{AND})_l^{con}$ are invalid under $\arj$.
\end{restatable}

See the Appendix for the proof.

%As a consequence of Proposition \ref{prop:rl}, this  implies that our blueprint consequence relation $\rj$ fails  $(AND)_l$ too.
\section{\urjt-consequence}

\label{sec:rrj}
Our next and final instantiation of the blueprint consequence relation $\rj$ arises by taking the more general approach to data-driven inference offered by the method of \textit{Strong Inference}, briefly recalled in Section \ref{sec:SI} above. 

Strong Inference stands in close analogy with the Ulam-Rényi game, which we spell out as the \textit{Ulam-Rényi game for Strong Inference} -- USI, for short. 

\subsection{Ulam-Rényi game for Strong Inference}	
A USI game is played by \emph{Scientists} (\textbf{S}) against \emph{Nature} (\textbf{N}). \textbf{N} ``thinks'' of a number, called the \emph{secret} which is the index of the unique true hypothesis in $\H=\{H_1, \ldots, H_n\}$.  \textbf{S} must figure out what the secret is, which we denote by $H_{s^*}$, and aims to do so as quickly as possible. The only move available to \textbf{S} is to ask Nature binary questions, i.e. questions that can be answered with either ``Yes'' or ``No''. In this latter case we say that \textit{the hypothesis has been {rejected}.}

%
%\begin{itemize}
%\item[] Two Players: \emph{Scientists} (\textbf{S}) and \emph{Nature} (\textbf{N})
%\item[]
%\item[] \textbf{N} ``thinks'' of a number, called the \emph{secret} which is the index of the unique true hypothesis in $\H=\{H_1, \ldots, H_n\}$
 
% \item[]
%\item[] \textbf{S} must figure out what the secret is, which we denote by $H_{s^*}$, and aims to do so as quickly as possible
% \item[]
%\item[]Yes-No questions
% \item[]
% \item[] \textbf{N} can lie  \textit{at most} $0 \leq m  \le \infty$ times
% \item[]
%\item[] $H$ is \emph{temporarily rejected} if  \textbf{N} answers ``No" to a question about $H$
 
% \item[]
%%\item[] $H$ is  \emph{rejected} if it has been {temporarily rejected} at least $m+1$ times.
%\end{itemize}
%

Ulam-R\'{e}nyi games are related to logic as follows. In the form just recalled above, they  provide a sound and complete semantics for classical logic. If however, Nature is allowed to \textit{lie} $m$ times ($m \in \mathbb{N}$), then USI game provides a sound and complete semantics for the $(m+2)$-valued \luka\;logic \cite{mundici1992logic,mundici1993ulam}.
	
Building on this, we will assume that \textbf{N} can lie  \textit{at most} {$0 \leq m  < \infty$} times. We say that a hypothesis $H$ is \emph{temporarily rejected} if  \textbf{N} answers ``No" to a question about $H$ and $m>0$. A hypothesis is   
  \emph{rejected} if it has been {temporarily rejected} at least $m+1$ times. %\footnote{\paolo{Ho controllato le sostituzioni e segnalato dove ho fatto modifiche. Ho solo un problema:usando  temporarily rejected mi pare ci sia conflitto con l'uso successivo di "least rejected hypothesis" che abbiamo spesso e a rigore dovrebbe ora essere "least temporarily rejected". Non potrebbe generare confusione?}}
%\comment{}{controllare che tutte le sostituzioni di "definitively rejected" con "rejected" e di "rejected con "temporarily rejected" siano andate a buon fine}

	Rather than a deceptive or mischievous view of Nature, this feature of USI games captures the fact that data are often gappy, ambiguous, or otherwise imperfect. Hence lies are interpreted here as mistakes made by scientists either in the design of experiments or in the analysis of the data generated by them. The fact that there is a bound to the number of mistakes captures the single most distinctive aspect of scientific reasoning: it will eventually self-correct. Indeed, the aim of the USI game is to  reject all candidates except for the secret, which  can be temporarily rejected, but not rejected.

The remainder of this work explores the consequences of taking the USI game as the intended semantics for our  \urjt-consequence whose properties are pinned down in the main result of this paper, Theorem \ref{thm:iminimal} below.

The formal setup is as follows. Each question-and-answer in a USI game is represented by a formula $\f$ in $\fm$, which we interpret as a positive answer to the question ``Does $\f$ hold?''.  We assume this is the only kind of data relevant to \urjt-consequence. So $\Delta$ is now a multiset from $\fm=\fm_\D=\fm_\H$. This means that data is  expressed in the same language of the hypotheses and we no longer distinguish $\fm_\D$ from $\fm_\H$. As described {in Section \ref{sec:quantitative}} 
%in the preparation for  blueprint consequence relation, 
scientists playing the USI game can rely on non-revisable assumptions, and in particular the fact that exactly one of the hypothesis is true (and Nature cannot lie about this). As usual we capture this with the formula $T$.

\begin{example}\label{ex:ulam} 
	Suppose that $\H = \{H_1,H_2, H_3, H_4\}$. Recall that  Scientists can only ask questions of the form  ``Does the secret belong to $D \subseteq \H$?". If, say,   $D=\{H_1\}$  then  $\mathbf{S}$ asks the question ``Is the secret $H_1$?". Suppose  Nature responds ``No". This means that the only hypothesis that has been temporarily rejected is $H_1$. The information provided by this first round of the USI game can be  expressed by the formula {$H_2 \lor H_3 \lor H_4$} and under the assumption that $T$ holds, this is logically equivalent to $\neg H_1$.  
	In general, if Nature's answer is ``Yes" and $I_{D}$ is the set of indexes of the hypotheses in $D$, then this information is expressed by the formula {$\bigvee_{i \in I_{D}} H_i$} which, modulo $T$, is logically equivalent to $\bigwedge_{i \in [n] \setminus I_D}\neg H_i$  ($n=|\H|$). If Nature's answer to the question ``Does the secret belong to $D \subseteq \H$?" is ``No", then this information is expressed by the two logically equivalent formulas {$\bigvee_{i \in [n] \setminus I_D}H_i$} and $\bigwedge_{i \in I_{D}} \neg H_i$. Table \ref{tab:form} illustrates sample questions-and-answers relative to $\H = \{H_1,H_2, H_3, H_4\}$.

\begin{table}[h!]
\begin{tabular}{l|c|cl}
Question & Answer & \multicolumn{2}{l}{{ Equivalent Formalizations (given $T$)}} \\ \hline
   $D=\{H_1, H_2\}$      &   Yes   &         ${H_1 \lor H_2}$          &      $\neg H_3 \land \neg H_4$            \\
      $D=\{H_4\}$     &    Yes    &        $H_4$          &     $\neg H_1 \land \neg H_2 \land \neg H_3$             \\
 $D=\{H_1, H_3, H_4\}$ &  No      &        $H_2$            &     $\neg H_1 \land \neg H_3 \land \neg H_4$          \\  
 $D=\{H_2\}$  &   No     &       {$ H_1 \lor  H_3 \lor H_4$}           &          $\neg H_2$      

\end{tabular}
\caption{Formalization of data in a USI game}
\label{tab:form}

\end{table}

\end{example}

 The \textit{degree to which a hypothesis is temporarily rejected by a single question-and-answer $\delta$} in a USI game after $k$ questions-and-answers, denoted $\rfu{}_{\d}$, is arrived at as follows. Before any question is answered, $\rfu{}_{\emptyset}(H)$ equals 0 for each $H\in \H$.  Then as data in the form of questions-and-answers arrive, let
	\begin{equation}\label{eq:rr}
{\rfu{}_{\d}(H)}= 
\begin{cases}
	\infty &\text{ if } T\models \neg \delta,\\
    1 &\text{ if } T, H \models\neg \d, \\
    0 &\text{ otherwise.}
\end{cases}
\end{equation}

As usual the aggregation of degrees of rejection yielded by multisets of data is additive, i.e. the sum of the degree of rejection of the formula occurring in the multiset. %So for  $H\in \H$ let 
%This is guaranteed to be finite if  $\rfu{}_\Delta(H)$ is not larger than $m$, the fixed number of lies in the USI game. 

There are however two exceptional cases where $\rfu{}_\Delta(H)$ is taken to be infinite, and different from the sum. The first arises when  
\( \sum_{\delta \in \Delta} \rfu{}_\delta(H_{s^*})>m \).
This accounts for the situation in which the secret has been rejected, i.e. temporarily rejected more times than the fixed number $m$ of available lies.
 Since \textbf{N} knows that the secret is true, this is a contradiction. (Recall that the blueprint consequence relation is explosive, see Lemma \ref{prop:rj-explosive}.)
The second case in which $\ru_\Delta(H_i)$ is infinite arises when hypotheses distinct from the secret are {rejected}, i.e. \( \sum_{\delta \in \Delta} \rfu{}_\delta(H_i)> m\). 

Summing up, we let:

\begin{equation}\label{infiniteru}
 \rfu{}_\Delta(H_i) = \infty
\begin{cases}
   \text { if  } H_i\in \H \text{ and }
 \sum_{\delta \in \Delta} \rfu{}_\delta(H_{s^*})>m\\
 \text { or } H_i\neq H_{s^*} \text{ and }
\sum_{\delta \in \Delta} \rfu{}_\delta(H_i){>}m
%\sum_{\delta \in \Delta} \ru_\delta(H) & \text{if  } H_i\in\H \text{ and } 
\end{cases}
\end{equation}
while, in the remaining cases:
\begin{equation}\label{eq:degrejulam}
 \rfu{}_\Delta(H_i)= \sum_{\delta \in \Delta} \rfu{}_\delta(H_i).
\end{equation}

Note that %before the last stage of the game, 
$H_{s^*}$ might not even belong to $\maxh_{\Delta}$ for arbitrary $\Delta$. We are only guaranteed that $\rfu{}_\Delta(H_{s^*}) \leq m$, i.e. it cannot be temporarily rejected more than  $m$ times.

%\comment{Da aggiungere}{Lemma: $\rfu{k}$ is a degree of rejection.}

\begin{lemma}\label{genincl}
{ $\rfu{}$ is a degree of rejection.}
\end{lemma}
  
\begin{proof}
We need to check that the rejection degree defined in \eqref{infiniteru} and \eqref{eq:degrejulam} satisfies Definition \ref{def:degrej}. 

For condition (1),  suppose $\c\models \d$, for $\c,\d\in\fm$. Since $\c\models\d$,  we have that $\neg \d\models \neg \c$, hence if $T\models \neg \d$, we have $T\models \neg \c$ and, if $T,H\models \neg \d$ then $T,H\models\neg \c$. This, by equation \eqref{eq:rr} entails that $\rfu{}_\c(H) \geq \rfu{}_\d(H)$.

The remaining conditions are immediate consequences of the definition, hence we omit checking them explicitly. 
% 
%As to condition (2) it is immediate to see that if $H\models \c$
%, then $P(t_H(\d)|H) = 1$, hence $r^\a_\d(H)=0$ for any $\a\in (0,1)$. \paolo{Similarly, for condition (3) if $H\models \neg d$, then $P( t(\d)|H) = 0 < \a$, hence $r^\a_\d(H)=1>0$.  
%Finally, condition (4) is obvious by definition.}
\end{proof}

\begin{example}\label{ex:ulam1}
Continuing with the setting of Example \ref{ex:ulam}, let $\H = \{H_1,H_2, H_3, H_4\}$. Suppose now that the number of lies Nature can use is $m=2$. Then, the following is a possible sequence of plays of the USI game:
\begin{enumerate}
\item Question: ``Is the secret in $D_{1}=\{H_1, H_2\}$?" \\
Answer: \; ``No".
\item Question: ``Is the secret in $D_{2}=\{H_4\}$?" \\
Answer: \; ``Yes".
\item Question: ``Is the secret in $D_{3}=\{H_1, H_2, H_3\}$?" \\
Answer: \; ``Yes".
\item Question: ``Is the secret in $D_{4}=\{H_1, H_2, H_4\}$?" \\
Answer: \; ``No".
\item Question: ``Is the secret in $D_{5}=\{H_4\}$?" \\
Answer: \; ``Yes".
\item Question: ``Is the secret in $D_{6}=\{H_4\}$?" \\
Answer: \; ``No".
\end{enumerate}

Denote with $\delta_i$ the formalization of the $i$-th question-and-answer, i.e.\\
\medskip
\begin{minipage}{\columnwidth/2}
 \begin{itemize}
 \item[]
 \item[] $\delta_1=H_3 \lor H_4$;
  \item[] $\delta_2=H_4$;
 \item[]$\delta_3=H_1\lor H_2 \lor H_3$;
\end{itemize}
\end{minipage}\hfill % maximize the horizontal separation
\begin{minipage}{\columnwidth/2}
 \begin{itemize}
  \item[]
 \item[]  $\delta_4=H_3$;
  \item[] $\delta_5=H_4$;
 \item[] $\delta_6=H_1 \lor H_2 \lor H_3$.
\end{itemize}
 \end{minipage}

Thus, the six moves in the USI game are captured formally by the multiset
\[ \Delta=\{H_3 \lor H_4, H_4, H_1 \lor H_2 \lor H_3, H_3,H_4,H_1 \lor H_2 \lor H_3\}.
\]
%Also for the rest of the paper, to enhance readability, the multisets representing the data will be indicated with single curly brackets.
Let us denote with $\Delta_i$ the multiset consisting of the formalizations of the first $i$ questions-and-answers according to the sequence just introduced, i.e. %\[\Delta_i=\{ \delta_j\, |\, 1\leq j \leq i\},\] and 
in our specific example:
\begin{itemize}
\item[] $\Delta_1=\{H_3 \lor H_4\}$;
\item[] $\Delta_2=\{H_3 \lor H_4, H_4\}$;
\item[] $\Delta_3=\{H_3 \lor H_4, H_4, H_1 \lor H_2 \lor H_3\}$;
\item[] $\Delta_4=\{H_3 \lor H_4, H_4, H_1 \lor H_2 \lor H_3, H_3\}$;
\item[] $\Delta_5=\{H_3 \lor H_4, H_4, H_1 \lor H_2 \lor H_3, H_3,H_4\}$;
\item[] $\Delta_6=\{H_3 \lor H_4, H_4, H_1 \lor H_2 \lor H_3, H_3,H_4,H_1 \lor H_2 \lor H_3\}$.
\end{itemize}
%\paolo{Non possiamo fare a meno di alcuni indici? Mi sembra un po' pesante da leggere l'esempio. Ad esempio, non mi è chiaro perché ci serve un indice distinto per numero di mosse, visto che è dato implicitamente dalla dimensione del multiset, e poi dopo ne facciamo comunque a meno.}
Therefore, after each question-and-answer we can compute the degree of rejection for each hypotheses in $\H$ according to \eqref{infiniteru} and \eqref{eq:degrejulam}. See Tables \ref{tab:exulam1} and \ref{tab:exulam2} for the computations of the degrees of rejection relative to this example. 
\begin{table*}[ht]
\begin{minipage}{\columnwidth/2}
    \centering
    \begin{tabular}{ccccc}
                       & \multicolumn{4}{c}{Hypotheses}         \\ \cline{2-5} 
\multicolumn{1}{c|}{}  & $H_1$ & $H_2$ & $H_3$ & \multicolumn{1}{c|}{$H_4$} \\ \hline
\multicolumn{1}{|c|}{$\rfu{}_{\delta_1}$} &  1  &  1  & 0   & \multicolumn{1}{c|}{0}   \\[5pt]
\multicolumn{1}{|c|}{$\rfu{}_{\delta_2}$} & 1   &  1  &  1  & \multicolumn{1}{c|}{0}   \\[5pt]
\multicolumn{1}{|c|}{$\rfu{}_{\delta_3}$} &  0  &  0  &  0  & \multicolumn{1}{c|}{1}   \\[5pt]
\multicolumn{1}{|c|}{$\rfu{}_{\delta_4}$} &  1  &  1  &  0  & \multicolumn{1}{c|}{1}   \\[5pt]
\multicolumn{1}{|c|}{$\rfu{}_{\delta_5}$} &  1  &   1 & 1   & \multicolumn{1}{c|}{0}   \\[5pt]
\multicolumn{1}{|c|}{$\rfu{}_{\delta_6}$} &   0 & 0   &  0  & \multicolumn{1}{c|}{1}   \\[5pt]\hline
\end{tabular}
\caption{Degrees of rejection computed on the single questions-and-answers $\delta_i$}
\label{tab:exulam1}
    \end{minipage}\hfill % maximize the horizontal separation
\begin{minipage}{\columnwidth/2}
    \centering
\begin{tabular}{ccccc}
                       & \multicolumn{4}{c}{Hypotheses}         \\ \cline{2-5} 
\multicolumn{1}{c|}{}  & $H_1$ & $H_2$ & $H_3$ & \multicolumn{1}{c|}{$H_4$} \\ \hline
\multicolumn{1}{|c|}{$\rfu{}_{\Delta_1}$} &  1  &  1  & 0   & \multicolumn{1}{c|}{0}   \\[5pt]
\multicolumn{1}{|c|}{$\rfu{}_{\Delta_2}$} & 2   &  2  &  1  & \multicolumn{1}{c|}{0}   \\[5pt]
\multicolumn{1}{|c|}{$\rfu{}_{\Delta_3}$} &  2  &  2  &  1  & \multicolumn{1}{c|}{1}   \\[5pt]
\multicolumn{1}{|c|}{$\rfu{}_{\Delta_4}$} &  $\infty$  &  $\infty$  &  1  & \multicolumn{1}{c|}{2}   \\[5pt]
\multicolumn{1}{|c|}{$\rfu{}_{\Delta_5}$} &  $\infty$  &   $\infty$ & 2   & \multicolumn{1}{c|}{2}   \\[5pt]
\multicolumn{1}{|c|}{$\rfu{}_{\Delta_6}$} &   $\infty$ & $\infty$   &  2  & \multicolumn{1}{c|}{$\infty$}   \\[5pt]\hline
\end{tabular}
\caption{Degrees of rejection computed on the sequence of questions-and-answers $\Delta_i$ }
\label{tab:exulam2}
 \end{minipage}
\end{table*}
Since $\rfu{}_{\Delta_6}(H_i)=\infty$ for $i=1, 2, 4$ and $\rfu{}_{\Delta_6}(H_3)=2$, i.e. all the hypotheses except for one have been rejected, it follows that $H_3$ is the secret. 
\end{example}

This example motivates our next Definition.

\begin{definition}[\urjt-consequence]\label{def:urj}
Let $\Delta$ be a multiset in $\fm_\H$,  $\f\in \fm_\H$, and $\rfu{}_{\Delta}(\f)$ be defined as in \eqref{infiniteru} and \eqref{eq:degrejulam}.
%, with \paolo{$r_\Delta(H_{s^*}) \leq m$ and $r_\Delta(H_i) \leq m+1$ for each $H_i \neq H_{s^*}$}. 
Then define
$$\Delta \urj \f\,  \text{ if }  \,T, H \models\f,\; \text{for every}\; H \in \maxh_{\Delta}. $$
\end{definition}
%The preconditions guarantee that $H_{s^*}$ cannot be rejected more than $m$ times and that the other hypotheses are not rejected more than $m+1$ times. Recall that $m$ is the maximum number of lies in the USI game.

Note that the fact we restrict data to answers in a USI game, effectively means that \urjt-consequence relates hypotheses.

A continuation of the Example \ref{ex:ulam1} illustrates Definition \ref{def:urj}. 

\begin{example}
Recall that $\H = \{H_1,H_2, H_3, H_4\}$, Nature can lie at most $m=2$ times, and the available data is formalized by the multiset $\Delta=\{H_3 \lor H_4, H_4, H_1 \lor H_2 \lor H_3, H_3,H_4,H_1 \lor H_2 \lor H_3\}$. For each $\Delta_i$ each  $\maxh_{\Delta_i}$ is computed as follows:\\
\begin{minipage}{\columnwidth/2}
 \begin{itemize}
 \item[] $\maxh_{\Delta_1}=\{H_3, H_4\}$;
  \item[] $\maxh_{\Delta_2}=\{H_4\}$;
 \item[]$\maxh_{\Delta_3}=\{H_3, H_4\}$;
\end{itemize}
\end{minipage}\hfill % maximize the horizontal separation
\begin{minipage}{\columnwidth/2}
 \begin{itemize}
 \item[]
 \item[]  $\maxh_{\Delta_4}=\{H_3\}$;
  \item[] $\maxh_{\Delta_5}=\{H_3, H_4\}$;
 \item[] $\maxh_{\Delta_6}=\{H_3\}$.
 \item[]
\end{itemize}
 \end{minipage}
Note that for every $i$, $j$ s.t. $i \not = j$ we have $T, H_i \models \neg H_j$. Therefore, the following consequence relations holds
\begin{itemize}
\item[]$\Delta_1 \urj \neg H_1 \land \neg H_2$;
\item[]$\Delta_2 \urj \neg H_3$;
\item[]$\Delta_6 \urj \neg H_1 \land \neg H_2 \land \neg H_4$.
\end{itemize}
{On the other hand, we have 
\begin{itemize}
\item[]$\Delta_1 \not\urj  \neg H_3$;
\item[] $\Delta_1 \not\urj  \neg H_4 $;
\item[]$\Delta_2 \not\urj \neg H_4$; 
\item[]$\Delta_6 \not \urj \neg H_3$.
\end{itemize}
}
%\comment{Aggiungere}{uno o due non-esempi}
\end{example}
Our next lemma collects immediate consequences of Definition \ref{def:urj} which will be useful later on.
\begin{lemma}
Let $\H=\{H_1, \dots, H_n\}$ and let $m \geq 0$ be the number of lies allowed to Nature. The following hold:
\begin{enumerate}
\item Suppose that $\maxh_{\Delta}=\maxh_{\Delta'}$, then $\Delta \urj \f$ if and only if $\Delta' \urj \f$.
\item If $\Delta$ rejects $H_i$, then $\Delta \urj \neg H_i$.
\item If $\Delta$ rejects $H_i$, then $\Delta' \urj \neg H_i$ for every $\Delta' \supseteq \Delta$ 
\item $H_{s^*}$ is the secret if and only if for some $\Delta \in \fm$ consistent with $T$, we have $\maxh_{\Delta}=\{H_{s^*}\}$ and $\rfu{}_\Delta(H_i) = \infty$ for all $H_i \not = H_{s^*}$.
\end{enumerate}
\end{lemma}

To show that indeed $\urj$ arises from the blueprint $\rj$ consequence relation we need to formulate a slightly stronger version of the rules (RWE) and (LLE)  where the theory $T$ is mentioned explicitly. (Recall $T$ expresses that the hypotheses form a partition containing the unique \textit{secret}). The resulting amendments are presented in Table \ref{tab:ruleswiththeory}, which also contains the strengthening of $\text{(AND)}_l$ and $\text{(AND)}_l^{con}$.
Note that, with a slight abuse of notation, we do not change the name of the rules compared to Table \ref{tab:mr}.

\begin{table}[h!]
\begin{tabular}{cc}
 \infer[\text{(RWE)}]{\Delta \nm \p} {\Delta \nm \f & &T,\f\models\p} & 
    \infer[\text{(LLE)}]{\Delta,\d\nm\f}{\Delta,\c\nm \f & &T\models\c\leftrightarrow\d } \\[2.0em]
 \infer[\text{(AND)}_l]{\Delta,\c\land\d\nm \f } {\Delta,\c,\d \nm\f & T,\neg\c,\neg \d \models \bot } & \infer[\text{(AND)}_l^{con}]{\Delta,\c,\d \nm\f } { \Delta,\c\land\d\nm \f &  T,\neg\c,\neg \d \models \bot }  \\
\end{tabular}
\caption{Rules for mutually exclusive and exhaustive hypotheses}\label{tab:ruleswiththeory}
\end{table}

% \infer[\text{(AND)}]{\Delta \rj   \f \land \p } {\Delta\rj  \f & & \Delta \rj  \p}  

As an immediate consequence of Proposition \ref{genmon} above, both (AMON) and (MMON) fail for $\urj$, as expected. Moreover a  straightforward adaptation of the arguments provided by the proof of Lemma \ref{lemma:rj-soundness} and Lemma \ref{lemma:prj-soundness} will yield the following Lemma.

\begin{lemma}\label{lemma:urj-soundness}
$\urj$ satisfies the rules in Table \ref{tab:ruleswiththeory}.
\end{lemma}

In light of this we can observe that $\urj$ arises from $\rj$ by taking $r=\rfu{}$. To simplify notation in the remainder of this Section we will just write $r$ and speak of the \textit{degree of rejection}.

\subsection{Valid rules of inference for $\urj$}
Since its language allows for a formula to be both on the left and on the right of the consequence relation symbol,  we can ask whether $\urj$ satisfies (multiset versions of)  (CUT), (CMO), and (RMO) described  in Table  \ref{tab:urj}. Quite interestingly, $\urj$ also satisfies a form of constrained monotonicity which to the best of our knowledge has not been investigated before, and which we term (UMO), for USI games:

\[\infer[(\umon)]{\Delta,\f\nm\p}{\Delta\nm \p &  \{\Delta\setminus\{\delta\}\nm  \p\}_{\delta\in\Delta. }}    \]

Before proving its validity, we illustrate the  rationale for (\umon). The non-monotonic behaviour of $\urj$ is ultimately due to the fact that new questions-and-answers can change the set of least rejected hypotheses. Note however that a \emph{single} question-and-answer may alter the degree of rejection of each hypothesis by at most by one. Hence,  when the degree of rejection of any least rejected hypothesis, say $H$, is at least two units smaller than the others in $\H$, we can be sure that adding just one question-and-answer, say $\f$, will not change the status of $H$ as a least rejected hypothesis. This condition on the ``distance'' between degrees of rejection of least rejected hypotheses and the others is captured by the second premise of (\umon). If the degree of rejection of the least rejected hypotheses are smaller by at the least two units than the others in $\H$, then removing any occurrence of a formula from the multiset $\Delta$ will not change $\maxh_{\Delta}$,  and hence the formulas that the multiset entails. 

\begin{example}
Suppose
\[ H_1,H_1\, \urj\,H_1.\]
Here $\maxh_{\{H_1,H_1\}}=\{H_1\}$, {$r_{\{H_1,H_1\}}(H_1) =0 $} and $r_{\{H_1,H_1\}}(H_i) = 2$ for any $H_i\in \H$, with $i\neq 1$. This ensures that \[H_1,H_1,\f\, \urj\, H_1 \]
for any formula $\f$, since  we may have, at worst,
$r_{\{H_1,H_1,\f\}}(H_1) =1$, while $r_{\{H_1,H_1,\f\}}(H_i)\geq 2$ for any $i \neq 1$.  $H_1$ will therefore continue to be the only formula in $\maxh_{\{H_1,H_1,\f\}}$.
The above reasoning is thus captured by the following application of (\umon): %(where we omit the condition on the secret, for simplicity):
\[\infer{H_1,H_1,H_2 \, \urj\, H_1}{H_1,H_1\, \urj\, H_1 & H_1\, \urj\, H_1 } \]

For a non-example, note that 
\[\infer{H_1,H_2\, \urj\, H_1}{ H_1\, \urj\, H_1}\]
is not valid.  To have a correct instance of (\umon), we should also have $\urj H_1$, which clearly does not hold. In the above consequence  we have $H_1\in \maxh_{H_1}$, but the difference between  $r_{H_1}(H_2)=1$   and $r_{H_1}(H_1) = 0$ is not strictly greater  than 1. In the consequence we now have  $r_{\{H_1,H_2\}}(H_1) = r_{\{H_1,H_2\}}(H_2) =1$, and both $H_1$ and $H_2$ belong to $\maxh_{\{H_1,H_2\}}$, hence {$\displaystyle{H_1,H_2\not\urj H_1}$}.

\end{example}

%us finally recall the form of the cut rule
%\[\infer[\text{(CUT)}]{\Delta \rj \p}{ \Delta,\f\rj \p & \Delta\rj\f}\]

%All of the considered rules are shown to hold in the next lemma.

\begin{table}
\begin{tabular}{cc}
\infer[\text{(REF)}]{\f\rj\f}{} & \\[2.0em]
\infer[\text{(CMO)}] {\Delta,\f \urj \p}{\Delta\urj \p &  \Delta\urj\f }
& \infer[\text{(CUT)}]{\Delta \rj \p}{ \Delta,\f\rj \p & \Delta\rj\f} \\[2.0em]
\infer[(\rmon)]{\Delta,\f\nm\p}{\Delta\nm \p & \Delta\not\nm \neg \f}
&
\infer[(\umon)]{\Delta,\c\nm\p}{\Delta\nm \p &  \{\Delta\setminus\{\delta\}\nm  \p\}_{\delta\in\Delta }} \\[2.0em]
\end{tabular}
\caption{Rules satisfied by the System $\urj$}
\label{tab:urj}
\end{table}

 \begin{restatable}{proposition}{Fifthprop}\label{soundmr}
%\emph{(CUT)}, \emph{(\rmon)} and \emph{(\umon)} hold for 
$\urj$ satisfies the rules in Table \ref{tab:urj}.
\end{restatable}
%\comment{}{In appendice. ho messo CMO in questa proposizione perché era un po' fuorviante avere la tabella delle regole soddisfatte da URJ senza che CMO ci apparisse, visto che è centrale}

Thus $\urj$ is \lq\lq more monotonic" than  rational consequence relations. So it captures a rather minimal departure from classical deduction in which non monotonic behaviour arises as a consequence of lies in the USI game, which corresponds to a variety of errors that scientists can make in experimental research. This vicinity with classical deduction can be interpreted as data-driven reasoning being a good approximation of the  gold standard of mathematical proof.

\subsection{Disjunction in the premisses and Modus tollens hold for $\urj$}\label{sec:MT}
In addition to satisfying (UMO) and (RMO), $\urj$ goes beyond $\prj$ by satisfying also the $\text{(OR)}$ rule. In light of our comments on Lemma \ref{failsxor} this welcome feature of $\urj$ should not come as a surprise upon careful comparison of \eqref{eq:s} and \eqref{eq:rr}. Whilst the semantics governing $\prj$ is quantitative (i.e. the full range of probability functions plays a role in \eqref{eq:s}), $\urj$ is qualitative (i.e. only the extreme points of   $[0,1]$ play any role in \eqref{eq:rr}).

\begin{restatable}{lemma}{Firstlemma}\label{lemma:vee}
$\urj$ satisfies
 \[\infer[\emph{(OR)}]{\Delta,\c\lor\d\nm\p.}{\Delta,\c\nm\p & \Delta,\d\nm\p}\]
 \end{restatable}
See the Appendix for the proof.

The qualitative nature of the  degree of rejection underpinning  $\urj$ gives us also a version of \textit{modus tollens}. This conspicuous rule, which to some is the key pattern in scientific reasoning, typically fails for quantitative inference. A fact that has not gone unnoticed to some critics of NHST, as we recalled in Section \ref{sec:NHST} above. Hence the following Lemma, proven in the Appendix, is of  conceptual interest.

\begin{restatable}{lemma}{Secondlemma}\label{mt}
$\urj$ satisfies 
\[\infer[\emph{(MT)}]{\Delta \urj  \neg \varphi.}{\Delta, \varphi \urj \psi & \Delta \urj  \neg \psi}
 \]
\end{restatable}

%\paolo{{On the other hand, $\urj$ does not satisfy the rule \emph{$(\text{\emph{NEG}}_n)$}, which characterize $\arj$. This can be seen by  applying the same argument for failure of \emph{$(\text{\emph{NEG}}_n)$} in $\prjt$ within the proof of Proposition \ref{negn}.}}

We now turn to the main result of this paper.

\subsection{A completeness result for $\urj$}
We make our way now to establishing a completeness result for \urjt-consequence relations. In preparation for that, the next Lemma pins down a useful normal form for arbitrary consequences under $\urj$. We will show that whenever  $\urj$ relates a multiset of premises and a conclusion, we can find an equivalent set of consequence relations $\urj$, with ``simpler" premises and conclusion. This will be useful in establishing the argument by cases which intervenes in the proof of Theorem \ref{thm:iminimal}. 

\begin{lemma}\label{normalform}%\footnote{\hykel{Mi pare che così si più facile da leggere}}

Let $\Delta$ be a multiset of formulas in $\fm$ and $\p\in \fm $. The following are equivalent:
  \begin{enumerate}
  \item $\Delta\urj\p$;
\item  A set of consequence relations of the form $\neg H_1^{l_1}, \dots, \neg H_n^{l_n} \urj \neg H_j$ holds, with  multiplicity indices $l_1, \dots, l_n$ greater or equal than $0$ and $j\in \{1,\dots,n\}$.
\end{enumerate}
\end{lemma}
\begin{proof}
(1) $\Rightarrow$ (2). Assume $\Delta \urj \p$. We proceed as follows.
First, note that any formula $\f$ in $\fm_\H$ is logically equivalent, modulo $T$, 
to the disjunction of the hypotheses entailing it, i.e. to the formula
\[\f_\lor:= \bigvee_{ H \models \f}  H.\]
In turn, since {under our assumptions} $  H\not \models \f$ entails $  H\models\neg \f$, the formula $\f_\lor$ is equivalent to :

\[ \f_\land := \bigwedge_{H \models \neg\f} \neg H.\] 
In the limiting case where  $\{H\in \H\mid H \models \f\} = \H $,
 and consequently  $\{H\in \H\mid H \models \neg\f\} = \emptyset$, we let $\f_\lor \equiv \f_\land \equiv T$. We then let $\Delta_\land$ be the multiset obtained by replacing each $\d\in\Delta$ by $\d_\land$. Since all the formulas in $\Delta_\land$ and $\p$ { are logically equivalent (modulo $T$) to  $\Delta$ and $\p$, %\footnote{\hykel{ non capisco quello che segue del paragrafo fino a RWE}}
 we will then have that $\Delta\urj\p$ if and only if 
\[\Delta_\land \urj \p_\land. \]}  From this latter, we can derive the original $\Delta \urj \p$ by repeated backwards applications  of (LLE) to recover $\Delta$, and (RWE) to recover $\p$
%\footnote{Qui però ci serve una versione più forte delle regole  LLE e RWE, che usi la teoria $T_\Delta$}.

Now, if one of the formulas in $\Delta$, say  $\d_\land$, is equivalent to $ T$, it does not {temporarily}
 reject any noncontradictory formula, hence \[\Delta_\land \urj \p_\land \text{ if and only if } \Delta_\land\setminus\{T\} \urj \p_\land.\] 
From it we can derive $\Delta_\land \urj \p_\land$ by just applying $(\rmon)$, since \linebreak {${{\Delta_\land\setminus\{T\} \not \urj \neg T}}$}, i.e.
\[ \infer[(\rmon)]{\Delta_\land \urj \p_\land }{\Delta_\land\setminus\{T\} \urj \p_\land  & \Delta_\land\setminus\{T\} \not \urj \neg T  }\]
We may thus remove from $\Delta$ all the formulas that are logically equivalent to $T$.

Our next step is applying backwards the rules $\text{(AND)}_l$ and $\text{(AND)}$. Note that we can equivalently take the converse $\text{(AND)}_l^{con}$, since the conjuncts are of the form $\neg H_1 \land \neg H_2 \land \dots$, and their negations $H_1,H_2,\dots$ are all mutually exclusive modulo T, hence they satisfy the condition for applying $\text{(AND)}_l^{con}$. We thus reduce \mbox{$\Delta_\land\urj \p_\land$} to the set of consequences of the form:
\[\neg H_1^{l_1}, \dots, \neg H_n^{l_n} \urj \neg H_j\] 
as required.

(2) $\Rightarrow$ (1). The result follows by applying (forwards, this time) all the rules used in the previous step.
\end{proof}

The normal form  depends on the set of hypotheses $\H$, as the next example illustrates.
\begin{example}\label{ex:NF} Consider first  the normal form of $H_1 \lor H_2, H_1 \lor H_3{\urj} H_1$ assuming $\H = \{H_1,H_2, H_3\}$:

\begin{itemize}
 \item[] $\d_1 \coloneqq  H_1 \lor H_2$
 \item[] $\d_2 \coloneqq  H_1 \lor H_3$
 \item[] $\psi \coloneqq  H_1$
 \end{itemize}
 
 \begin{itemize}
 \item[] ${\d_1}_{\land} \coloneqq \neg H_3 $
 \item[] ${\d_2}_{\land} \coloneqq   \neg H_2$
 \item[] $\psi_{\land} \coloneqq  \neg H_2 \land \neg H_3$
 \end{itemize}
 
Thus, $\Delta_{\land}=\{\neg H_3, \neg H_2\}$ and $\maxh_{\Delta_{\land}}=\{H_1\}$.

Note that if we let  $\H = \{H_1,H_2, H_3, H_4\}$ we obtain the following:
 
 \begin{itemize}
 \item[] ${\delta'_1}_{\land} \coloneqq \neg H_3 \land \neg H_4 $
 \item[] ${\delta'_2}_{\land} \coloneqq   \neg H_2 \land \neg H_4$
 \item[] $\psi'_{\land} \coloneqq  \neg H_2 \land \neg H_3 \land \neg H_4$
 \end{itemize}

Hence $\Delta_{\land}'=\{\neg H_3, \neg H_2, \neg H_4\}$ and $\maxh_{\Delta'_{\land}}=\{H_1\}$.

It is easy to see that, in case $\H = \{H_1,H_2, H_3, H_4\}$ we have $\Delta'_{\land}\not\urj H_1$. 
\end{example}

Now to our main result.

\begin{restatable}{theorem}{Firsttheorem}\label{thm:iminimal}
Suppose $\Delta\urj\p$ and $r_\Delta(H)$ is finite for each $H\in\H$. Then there is a derivation of it using the rules in Table \ref{tab:ruleswiththeory} and Table \ref{tab:urj}.
\end{restatable}

The proof is in the Appendix. Here we sketch the idea. 
Lemma \ref{normalform} guarantees the existence of a set of
\[\neg H_1^{l_1}, \dots, \neg H_n^{l_n} \urj \neg H_j,\] 
%one for every $H$ s.t. $H \models \psi$, 
where all the multiplicity indices $l_1, \dots, l_n$  are greater or equal than $0$ and $j\in \{1,\dots,n\}$.
Pick any such consequences and denote its left-hand side by $\Delta'$.
The proof then proceeds reasoning by induction on the number of formulas occurring in $\Delta'$. We need to show how to derive $\Delta' \urj \neg H_j$ by a rule in either Table \ref{tab:ruleswiththeory} and  Table \ref{tab:urj}, using valid  premises with a number of formulas smaller than $\Delta'$.
Example \ref{ex:thm} below illustrates some of the relevant subcases, involving  (\umon) and (\rmon).

%\remove{Recall that the above proofs only focused on consequences of the form $\Delta \urj \p$ with $r_\Delta(H)$ finite for each $H\in \H$. This means that we are focusing on the $\Delta$ corresponding to those part of the USI game which are still fully non-monotonic, i.e. when none of the hypothesis have been rejected yet. There is in principle no major problem in extending the proof to the case where for \emph{some} of the hypotheses $H$, we have $r_\Delta(H) = \infty$, but this would burden an already technical proof with further case distinctions. On the other hand the case where $r_\Delta(H) = \infty$ for \emph{all} of the hypothesis, i.e. when $r_\Delta(H_{s^*}) \geq m+1$ is not covered by our proof, but has little interest, since the  multiset $\Delta$ would be an illegal one in the corresponding USI game.}

\begin{example}\label{ex:thm}
We illustrate by way of examples how to apply the rules (\umon) and (\rmon) in the corresponding sub-case of the proof of Theorem \ref{thm:iminimal}. In each case below, we set $\H=\{H_1, H_2, H_3, H_4\}$. 
\begin{itemize}
\item[(1)] Suppose that $\Delta'=\{\neg H_4^3\}$ where 3 is the multiplicity index relative to $\neg H_4$. Thus, 
$\maxh_{\Delta'}=\H\setminus\{H_4\}=\{H_1, H_2, H_3\}$, and$\neg H_4^3 \urj \neg H_4$ holds, while $H_i \models \neg H_4$, for every $i=1, 2, 3$.  
If we remove $\neg H_4$ from $\Delta'$, then 
$\Delta'\setminus \{\neg H_4\}=\{\neg H_4^{2}\}$ and 
$\maxh_{\Delta'}=\maxh_{\Delta'\setminus \{\neg H_4\}}=\{H_1, H_2, H_3\}$. Therefore, $\Delta'\setminus \{\neg H_4\} \urj \neg H_4$. In addition, for every $i=1, 2, 3$ we have $H_i \not \models H_4$ and $\Delta'\setminus \{\neg H_4\} \not\urj H_4$.

 In conclusion $\neg H_4^3 \urj \neg H_4$ can be obtained by applying (\rmon) to $\neg H_4^2 \urj \neg H_4$ as first premise and $\neg H_4^2 {\not\urj} H_4$ as second premise, i.e. 

\[ \infer[{(\rmon)}]{\neg H_4^3 \urj \neg H_4} {\neg H_4^2 \urj \neg H_4  & \neg H_4^2 \not\urj H_4  }.\]

Suppose that $\Delta'=\{\neg H_1^2, \neg H_4^3\}$. Thus, 
$\maxh_{\Delta'}=\{H_2, H_3\}$. Since $H_2\models \neg H_1$ and $H_3\models \neg H_1$, then
$\neg H_1^2, \neg H_4^3 \urj \neg H_1$ holds. If we remove $\neg H_4$ from $\Delta'$, then 
$\Delta'\setminus \{\neg H_4\}=\{\neg H_1^2, \neg H_4^2\}$ and 
$\maxh_{\Delta'}=\maxh_{\Delta'\setminus \{\neg H_4\}}=\{H_2, H_3\}$. Therefore, $\Delta'\setminus \{\neg H_4\} \urj \neg H_4$. In addition, for every $i=2, 3$ $H_i \not\models H_1$ and $\Delta'\setminus \{\neg H_4\} \not\urj H_1$.

In conclusion $\neg H_1^2, \neg H_4^3 \urj \neg H_1$ can be obtained by applying (\rmon) to $\neg H_1^2, \neg H_4^2 \urj \neg H_1$ as first premise and\\ $\neg H_1^2, \neg H_4^2\not\urj H_4$ as second premise, i.e. 

\[ \infer[{(\rmon)}]{\neg H_1^2, \neg H_4^3 \urj \neg H_1} {\neg H_1^2, \neg H_4^2 \urj \neg H_1  & \neg H_1^2, \neg H_4^2 \not\urj H_4 }.\]

\item[(2a)] Suppose that $\Delta'=\{\neg H_1^2, \neg H_2^1, \neg H_3^3, \neg H_4^3\}$.
Thus, $\maxh_{\Delta'}=\{H_2\}$. Since $H_2\models \neg H_4$, 
then $\neg H_1^2, \neg H_2^1, \neg H_3^3, \neg H_4^3 \urj \neg H_4$ holds.
If we remove $\neg H_3$ from $\Delta'$, then 
$\Delta'\setminus \{\neg H_3\}=\{\neg H_1^2, \neg H_2^1, \neg H_3^2, \neg H_4^3\}$ and 
$\maxh_{\Delta'}=\maxh_{\Delta'\setminus \{\neg H_3\}}=\{H_2\}$. 
Therefore, $\Delta'\setminus \{\neg H_3\} \urj \neg H_4$. In addition, $H_2\not \models H_3$ and $\Delta'\setminus \{\neg H_3\} {\not\urj} H_3$. 

In conclusion $\neg H_1^2, \neg H_2^1, \neg H_3^3, \neg H_4^3 \urj \neg H_4$ can be obtained by applying (\rmon) to 
$\neg H_1^2, \neg H_2^1, \neg H_3^2, \neg H_4^3 \urj \neg H_4$ as first premise and\\ $\neg H_1^2, \neg H_2^1, \neg H_3^2, \neg H_4^3 {\not\urj}  H_3$ as second premise, i.e. 

\[ \infer[{(\rmon)}]{\neg H_1^2, \neg H_2^1, \neg H_3^3, \neg H_4^3 \urj \neg H_4} {\neg H_1^2, \neg H_2^1, \neg H_3^2, \neg H_4^3 \urj \neg H_4 & \neg H_1^2, \neg H_2^1, \neg H_3^2, \neg H_4^3 \not\urj H_3}.\]

\item[(2a')] Suppose that $\Delta'=\{\neg H_1^2, \neg H_2^1, \neg H_3^2, \neg H_4^3\}$.
Thus, $\maxh_{\Delta'}=\{H_2\}$. Since $H_2\models \neg H_4$, 
then $\neg H_1^2, \neg H_2^1, \neg H_3^2, \neg H_4^3 \urj \neg H_4$ holds.

If we remove $\neg H_2$ from $\Delta'$, then 
$\Delta'\setminus \{\neg H_2\}=\{\neg H_1^2, \neg H_3^2, \neg H_4^3\}$ and 
$\maxh_{\Delta'}=\maxh_{\Delta'\setminus \{\neg H_3\}}=\{H_2\}$. 
Therefore, $\Delta'\setminus \{\neg H_2\} \urj \neg H_4$. 

If we remove $\neg H_2$ and $\neg H_1$ from $\Delta'$, then 
$\Delta'\setminus \{\neg H_2, \neg H_1\}=\{\neg H_1, \neg H_3^2, \neg H_4^3\}$ and 
$\maxh_{\Delta'}=\maxh_{\Delta'\setminus \{\neg H_2, H_1\}}=\{H_2\}$. 
Therefore, $\Delta'\setminus \{\neg H_2, H_1\} \urj \neg H_4$. 

If we remove $\neg H_2^2$ from $\Delta'$, then 
$\Delta'\setminus \{\neg H_2^2\}=\{\neg H_1^2, \neg H_3^2, \neg H_4^3\}$ and 
$\maxh_{\Delta'}=\maxh_{\Delta'\setminus \{\neg H_2^2\}}=\{H_2\}$. 
Therefore, $\Delta'\setminus \{\neg H_2^2\} \urj \neg H_4$. 

If we remove $\neg H_2$ and $\neg H_3$ from $\Delta'$, then 
$\Delta'\setminus \{\neg H_2, \neg H_3\}=\{\neg H_1^2, \neg H_3, \neg H_4^3\}$ and 
$\maxh_{\Delta'}=\maxh_{\Delta'\setminus \{\neg H_2, H_3\}}=\{H_2\}$. 
Therefore, $\Delta'\setminus \{\neg H_2, H_3\} \urj \neg H_4$.

If we remove $\neg H_2$ and $\neg H_4$ from $\Delta'$, then 
$\Delta'\setminus \{\neg H_2, \neg H_4\}=\{\neg H_1^2, \neg H_3^2, \neg H_4^2\}$ and 
$\maxh_{\Delta'}=\maxh_{\Delta'\setminus \{\neg H_2, H_4\}}=\{H_2\}$. 
Therefore, $\Delta'\setminus \{\neg H_2, H_4\} \urj \neg H_4$. 

In conclusion $\neg H_1^2, \neg H_2^1, \neg H_3^2, \neg H_4^3 \urj \neg H_4$ can be obtained by applying (\umon) to 
$\neg H_1^2, \neg H_3^2, \neg H_4^3 \urj \neg H_4$ as first premise and $\{\{\neg H_1^2, \neg H_2^1, \neg H_3^2, \neg H_4^3\}\setminus \{\neg H_2, \neg H\} \urj \neg H_4\}_{\neg H \in \Delta'}$, i.e.

\[ \infer[{(\umon)}]{\neg H_1^2, \neg H_2^1, \neg H_3^2, \neg H_4^3 \urj \neg H_4} {\neg H_1^2, \neg H_3^2, \neg H_4^3 \urj \neg H_4 & \{\{\neg H_1^2, \neg H_3^2, \neg H_4^3\}\setminus \{ \neg H\} \urj \neg H_4\}_{\neg H \in \Delta'}}.\]

\item[(2b)] Suppose that $\Delta'=\{\neg H_1^2, \neg H_2^1, \neg H_3^2, \neg H_4^3\}$. Thus, 
$\maxh_{\Delta'}=\{H_2\}$. Since $H_2\models \neg H_1$, then
$\neg H_1^2, \neg H_2^1, \neg H_3^2, \neg H_4^3 \urj \neg H_1$ holds. If we remove $\neg H_4$ from $\Delta'$, then 
$\Delta'\setminus \{\neg H_4\}=\{\neg H_1^2, \neg H_2^1, \neg H_3^2, \neg H_4^2\}$ and 
$\maxh_{\Delta'}=\maxh_{\Delta'\setminus \{\neg H_4\}}=\{H_2\}$. Therefore, $\Delta'\setminus \{\neg H_4\} \urj \neg H_4$. In addition, $H_2 \not\models H_1$ and $\Delta'\setminus \{\neg H_4\} \not\urj H_1$. 

In conclusion $\neg H_1^2, \neg H_2^1, \neg H_3^2, \neg H_4^3 \urj \neg H_1$ can be obtained by applying (\rmon) to $\neg H_1^2, \neg H_2^1, \neg H_3^2, \neg H_4^2 \urj \neg H_1$ as first premise and \\$\neg H_1^2, \neg H_2^1, \neg H_3^2, \neg H_4^2\not\urj H_4$ as second premise, i.e. 

\[ \infer[{(\rmon)}]{\neg H_1^2, \neg H_2^1, \neg H_3^2, \neg H_4^3 \urj \neg H_1} {\neg H_1^2, \neg H_2^1, \neg H_3^2, \neg H_4^2 \urj \neg H_1& \neg H_1^2, \neg H_2^1, \neg H_3^2, \neg H_4^2\not\urj H_4}.\]

\end{itemize}
\end{example}

\begin{example} We further illustrate Theorem \ref{thm:iminimal} by showing how to find a derivation for the normal forms identified in Example \ref{ex:NF}. To this end,  let $\H = \{H_1,H_2, H_3\}$ and $\neg H_2, \neg H_3 \urj \neg H_2 \land \neg H_3$. Thus,
  \begin{small}
\[
\infer[\text{(AND)}]{\neg H_2, \neg H_3 \urj \neg H_2 \land \neg H_3} {\infer[(\rmon)]{\neg H_2, \neg H_3 \urj \neg H_2}{\neg H_2 \urj \neg H_2 & \neg H_2 \not \urj H_3} &\infer[(\rmon)]{\neg H_2, \neg H_3 \urj \neg H_3}{\neg H_3 \urj \neg H_3 & \neg H_3 \not \urj H_2}},
\]
\end{small}
is the required derivation.
\end{example}

\section{Conclusion and further work}\label{sec:conclusion}

We have introduced a family of consequence relations with the following intended semantics: \emph{$\phi$ follows from $\Delta$ if $\phi$ is a classical consequence of every hypothesis which is least rejected in $\Delta$.
} The formalisation of this leads to our blueprint consequence relation of Definition \ref{def:Rj}. This in turn produces three distinct, but related consequence relations: $\prj$, $\arj$ and $\urj$.  The first, $\prj$, is based on a  rejection function which pins down maximal likelihood hypotheses. The rejection function underpinning $\arj$ is based on a logical rendering of the  p-value.  Finally, the rejection function defining $\urj$ arises by counting of how many times a hypothesis is {temporarily} rejected in a USI game.

We showed that $\prj$, $\arj$, and $\urj$  are variations on the theme of {rational} consequence relations, as is their precursor: the blueprint consequence relation $\rj$.  Table \ref{tab:summary} provides a comparison (we omit from the table those rules which are satisfied by all of them, i.e. REF, LLE, RWE, AND,  CUT, {RMO},CMO).  %CMO,REF
\begin{table}[h!]
  \centering
 
  \begin{tabular}{l|cccccc}
     &   $\text{AND}_l$ & $\text{AND}_l^{con}$ & $\text{OR}$ &  $\umon$ \\   
\hline
%$\rj$   & & &  & &   \\
$\prj$   & \checkmark&\checkmark& &    &  \\   
$\arj$ & &  &  & &\\
$\urj$  & \checkmark & \checkmark & \checkmark & \checkmark   \\   
  \end{tabular}

%  \begin{tabular}{l|cccccc}
%     &   $\text{AND}_l$ & $\text{OR}$ & $\text{NEG}_n$ & $\rmon$ & $\umon$ \\   
%\hline
%$\rj$   & & &  & & &  \\
%$\prj$   & \checkmark&&  & & &  \\   
%$\arj$ & \checkmark& & \checkmark& \checkmark& &\\
%$\urj$    & \checkmark& \checkmark & &\checkmark  &\checkmark \\   
%  \end{tabular}
  \caption{A comparison of relevant properties satisfied by rejection-based consequence relations.}
  \label{tab:summary}
\end{table}

The main result of the paper, Theorem \ref{thm:iminimal}, identifies the rules of inference with respect to which $\urj$ is complete.

While laying down the intended semantics of our consequence relations, we detailed why we considered the rejection degrees desirable for maximum likelihood, null hypothesis significance tests, and strong inference, respectively. To us this lends support to the  blueprint consequence relation being a promising starting point for a logical investigation into the general properties of data-driven inference, independently of one's own preferred view on the nature of probability. 

Given the enormous variety and complexity of scientific inference, what we have put forward here is just a preliminary set of results. We hope that this is enough to encourage more logicians to take up the challenge. Not only will this help revamping the decidedly outdated picture scientists have of logic, but it may also contribute to bringing some of the most heated debates on statistical significance and, more generally, on the methodology of data-driven inference on more logical and less ideological grounds.

\subsection{Related and future work}\label{sec:related}
In a series of works culminating in \cite{Kyburg2006}, Henry Kyburg and Choh Man Teng investigate the problem of putting statistical inference on a non-monotonic logical footing. They do so by taking as a starting point  the defeasible nature of rejecting $H_0$ in NHST.
To the best of our knowledge they are among the very first ones {to do} this from a formal-logical point of view.  It is therefore appropriate to describe briefly how the present work departs from the pioneering contribution of Kyburg and  Teng's.

Conceptually,   \cite{Kyburg2006}  focusses on the fact that in NHST and, more generally in statistical inference, one usually justifies conclusions (about the population) based on the representativeness of the sample. So its aim, as far as the foundations of statistics is concerned, is to contribute {to} the long-standing \textit{reference class problem}. Our motivating questions is broader, as we pointed out in the introductory section. Specifically, whilst we acknowledge the paramount importance of statistical data in scientific inference,  we do not assume that it is the only kind of data relevant to pin down the logical properties of scientific inference. This is why we encompass the more general problem of \textit{strong inference}, which we capture through USI games. Our second, and more formal, point of departure from Kyburg and Teng is the specific non-monotonic framework.  Given their research question, Default Logic \cite{Reiter1981}  is a most natural setup, which they then expand to investigate the role of probabilistic evidence in providing a semantics for the justification of non-normal defaults.  More precisely, the authors insist that the justification for the application of a default rule encodes the lack of evidence to the effect that the statistical sample used in the inference is atypical or biassed.  In line with the  broader question asked, our results are formulated in terms of non-monotonic consequence relations of the preferential kind. As is well known \cite{Makinson2005} the two frameworks are related, but distinct.  Similarly, and more importantly, we do not commit to any specific view on the meaning of probability, which is interpreted evidentially by Kyburg and Teng. Since their uncertainty representation belongs to the wider class of imprecise probability, we are at  present  unable to carry out a direct comparison with their results. This, however, we set out to do in future work.

One problem which arises in the context of rejection-based scientific inference, is handling conflicting evidence. Kyburg and Teng tackle it at the level of the measure of uncertainty, which  for them is evidential probability as just recalled. However, it is of technical and conceptual interest to focus on the fact that rejection-based inference can {be} seen as a form of paraconsistent inference. Recall that according to Fisher \cite{Fisher1935a}:
\begin{quote}
[t]he interest of statistical tests for scientific workers depends entirely from their use in rejecting hypotheses which are thereby judged to be incompatible with the observations.
\end{quote}
As captured by USI games, if the observations are classically inconsistent with a set of hypotheses, at least one of them must be rejected. But on more fine-graned analysis, this kind of ``incompatibility'' will come in degrees. Whilst the consequence relation $\arj$ tackles this by mimicking the calculated p-value, it is interesting to address the ``degree of incompatibility'' interpretation of p-values  through a paraconsistent consequence relation, for which a vast landscape of logics are available \cite{Carnielli2002,Carnielli2016}. As a first line of work in this direction, it appears promising to distinguish \emph{evidence in favour} from \emph{evidence against} a given hypothesis. In this way a formal link with First-Degree Entailment \cite{Omori2017} becomes available. This would then lead to revising our blueprint consequence relation to encompass logics of formal inconsistency \cite{Carnielli2007,Carnielli2017}. In  particular, the probability functions defined over such systems  \cite{Carnielli2021,Klein2021} are expected to give rise to rejection degrees capable of representing useful inference in light of conflicting evidence, which is very much the norm in scientific practice. The statistical literature on this is already quite promising
  % This could also be done by building formal bridges with the theory of \emph{agnostic hypothesis testing}
\cite{borges2007rules,stern2004paraconsistent,Izbicki2015,Esteves2016}, see also \cite{Pereira2022} for an overview.

This brings us, in conclusion, to a logical take on statistical inference which has been put forward in the framework of dynamic doxastic logics by Baltag, Rafiee Rad and Smets (BRS) in \cite{Baltag2021}. In it, the authors model an agent performing statistical inference, as combining ``strong'', epistemic information about the probability of an event, with ``soft'', doxastic, information. The latter is construed as a plausibility function, which determines an ordering over the epistemically permissible probability distributions, reflecting the subjective inclination of the agent and the non-definitive evidence she is confronted with. The agent is then taken to believe only in the \emph{maximal} distributions, according to her plausibilistic order. Whilst the research question underlying the BRS framework is quite distinct from ours, a foundational commonality with the present work stands out: inference arises against some underlying epistemic ordering and background information. This is not surprising, since some epistemic ordering is always at the root of a significance test.

In conclusion, let us restate the key message of our results: consequence relations based on degrees of rejection are promising in the investigation of the validity of data-driven inference. A clear difficulty in making the next step, and ascertain whether they will also deliver in real-world cases of scientific inference, is to do with the fact that these  are only partially formal. By this we mean that unlike mathematical proofs, arguments in data-driven science depend, to some hard-to-pin-down extent to the actual content of the specific reasoning at hand. This is why we are working on a bottom-up approach where specific instances of data-driven inference which a relevant community takes to be valid is represented within our system of consequence relations and hopefully abstracted to show its general, yet context-dependent, validity. Key tools in doing this are the further specialisation of the rejection functions, and in particular alternatives to addition in the aggregation of the rejection offered by single pieces of data, and a more fine grained interpretation of the meaning of \textit{lies} in the USI game. We hope to report encouraging results in this direction in future work.

\section*{Acknowledgements}
We are very grateful to the organisers and audiences of the LIRa seminar in Amsterdam, the IIIA-CSIC seminar in Barcelona, the triennial international conference of the Italian Society for Logic and the Philosophy of Science held at the University of Urbino, {the MOSAIC Workshop held in Vienna}. 

Paolo Baldi acknowledges funding from the research project ``Green-Aware MEchanismS'' (GAMES), part of the PNRR MIUR project FAIR - Future AI Research (PE00000013), Spoke 9 - Green-Aware AI

%Esther Anna Corsi acknowledges funding from Fondazione Cariplo under the ``Young Researchers'' call -- Logics for Scientific Inferences (LOGSI) 2023-0978. 

Hykel Hosni acknowledges funding from the Italian Ministry for University and Research under the scheme FIS1 -- Reasoning with Data (ReDa) G53C23000510001. 

{Corsi and Hosni acknowledge funding by the Department of Philosophy ``Piero Martinetti'' of the University of Milan under the Project ``Departments of Excellence 2023-2027'' awarded by the Italian Ministry of Education, University and Research (MIUR), and by the MOSAIC project (EU H2020-MSCA-RISE-2020 Project 101007627).}
%the Italian Ministry for University and Research under 
\section*{Appendix}

\Firstprop*
  \begin{proof}
The proof proceeds by considering each rule in turn.

For (RWE), assume that $\Delta\rj\delta$ and $\delta \models \p$.  Since $\Delta\rj\delta$  for any $H\in  \maxh_\Delta$ we have that  
$T, H\models \delta$. From $\delta\models \p$, we then have, for any $H\in \maxh_\Delta$, that $T, H\models \p$. %Hence the rule is valid.

(LLE) holds, since if we take two formulas $\c$ and $\d$ which are logically equivalent, by (2) in Definition \ref{def:degrej} we have %$\maxh_\f = \maxh_\p$, and 
$r_\c(H) = r_\d(H)$ for any $H \in \H$. Hence $\maxh_{\Delta,\c}= \maxh_{\Delta,\d}$ and $\Delta,\c\rj \f$ entails $\Delta,\d\rj \f$ .

%For (REF), we have that $\f$ can only be both on the left and  right-hand side of $\rj$ if $\f\in\fm_\H$. Then for every $H\in \H$, either $H \models \f$ or $H\models \neg \f$. In the former case, we have $r_\f(H) = 0 $, while in the latter  $r_\f(H) = 1$. Hence, we have that $H\in \maxh_\f$ iff $H\models \f$. But this shows that $\f\rj\f$.

We now show that (AND) holds. Indeed, assume that ${\Delta\rj \f}$ and  ${\Delta\rj \p}$. For any $H\in \maxh_{\Delta}$, we then have that $T, H\models\f$ and $T, H\models\p$. Hence $T, H\models \f\land \p$ for any $H\in \maxh_{\Delta}$.\end{proof}

\medskip

\Secondprop*
\begin{proof}
	
	For (AMON), let  $\H = \{H_1,H_2\}$ %and suppose  $t$ is the identity function. 
	and take $\c\rj H_1$ to be the premise of (AMON), assuming that   $r_\c(H_1) \leq r_\c(H_2)$.  Now, we will have $r_{\c\land \d} (H_1) \geq r_{\c} (H_1)$ and $r_{\c\land \d} (H_2) \geq r_{\c} (H_2)$ by the properties of rejection functions. 
	%This means that %r_{\f\land \xi} (H_2) - r_{\f\land \xi} (H_1)\geq 
	However, we may still have  $r_{\c\land\d}(H_1) > r_{\c\land\d}(H_2)$, that is $\maxh_{\c\land\d} = \{H_2\}$, hence $\c\land\d\not\rj H_1$.
	
	%\begin{equation}\label{eq:lemma-1}
	%r_(\neg \f \mid H_2) + P( \neg \xi \mid H_2)< P(\neg \f \mid H_1) + P( \neg \xi \mid H_1) .
	%\end{equation}
	%Since $P(\neg(\f\land \xi)) = P(\neg \f\lor \neg\xi)= P(\neg \f) + P(\neg \xi)$,  \eqref{eq:lemma-1} is equivalent to $P(\neg(\f\land\xi)\mid H_1) > P(\neg (\f\land\xi) \mid H_2)$. 
	
	%Hence we have constructed a situation in which  $\f\land\xi \not\rj H_1$ and $\f\land\xi\rj H_2$, as required.
	
	As for (MMON), assume again $\H = \{H_1,H_2\}$
	and $\c\rj H_1$ with $r_\c(H_1) = 0$ and $r_\c(H_2) =1$. Let $\d$ be such that $r_\d(H_1) = 1$ and $r_\d(H_1) = 0$. We will have $r_{\c,\d}(H_1) = r_{\c,\d}(H_2) = 1$. Hence $\maxh_{\c,\d} = \{H_1,H_2\}$, and clearly $\c,\d \not\rj H_1$.\end{proof}
	
\medskip

\Thirdprop*
\begin{proof}
For (REF) we proceed as follows. First, note that from the assumptions on $T$ it follows that, for every $H\in \H$, either $T, H \models \delta$ or $T, H\models \neg \delta$. In the former case, we have $r_\delta(H) = 0 $ and $H\in\maxh_\delta$, while in the latter $r_\delta(H)>0 $, hence $H\not \in \maxh_\delta$.  In case that, at least for  some $H$, we have $T,H\models \delta$, we then immediately obtain, for each $H\in\maxh_\delta$ that  $T, H\models\delta$ holds.

Assume now instead that, for all $H\in \H$, we have $H,T\models \neg \delta$. Since the hypotheses are exhaustive and mutually exclusive, this entails that $T \models \neg \delta$, hence $r_\delta(H) =\infty$ for each $H\in \H$ and  by Lemma \ref{prop:rj-explosive}, we have that $\delta\rj\delta$.

% $\maxh=\emptyset$ and $\f\rj\f$ since its defining condition ``if $H\in\maxh_\f$ then $T,H\models \f$"  is vacuously true.

%Hence, by the assumption $T_\f\models \bot$ and for each $H\in\maxh$ we have $T_\f,H \models \bot$, which in turn entails 
%$T_\f, H \models\f$.
%Hence, by the assumptions, $r_\f(H) = \infty$ for each $H\in \H$, and $T_\f\models \neg H$. But this would entail that, for each $H\in\maxh_\f$, we have $H,T_\f\models\neg H$. Hence $H, T_\f $ is  contradictory and in particular $H, T_\f\models \f$ . 

%In case all of the $H$ are such that $T_\Delta, H\models \neg \f$, 

% Hence, we have that $H\in \maxh_\f$ iff $H\models \f$. But this shows that $\f\rj\f$.

For (CUT),  let us assume that (a) $\Delta,\delta\rj\p$ and (b) $\Delta\rj\delta$. %Note that,  since $\f$ occurs both on the right and the left hand side of $\rj$, we need to have $\f\in\fm_\H$.
If $\maxh_{\Delta} = \emptyset$ then by Lemma \ref{prop:rj-explosive}, $\Delta\rj \p$ holds. So  assume $\maxh_{\Delta}\neq \emptyset$ and let  $H \in \maxh_{\Delta}$. We need to show that $T,H \models \p$. From (b) it follows that $T, H \models \delta$
hence, $r_{\delta}(H)=0$. But the latter means that  $r_{\Delta,\delta}(H) = r_\Delta(H)+ r_\delta(H)= r_\Delta(H)$.
Hence, if $H\in \maxh_{\Delta}$, we also have $H \in\maxh_{\Delta,\delta}$. But by (a), this means that $H\models\p$, thus showing our claim. 

% From (a) it follows that, for every $H \in  \maxh_{\Delta,\f}$ we have $T_{\Delta,\f}, H \models \p$. From (b) it follows that for every $H' \in \maxh_{\Delta}$ we have $T_\f, H' \models \f$, 
% hence $r_{\f}(H')=0$ and  $r_{\Delta,\f}(H') = r_\Delta(H')+ r_\f(H')= r_\Delta(H')$

{For (RMO), assume that $\Delta\rj\p$ and $\Delta \not\rj\neg\f$. Let us denote by  $H_\f$  all and only the hypothesis in $\H$ that entail  $\f$. %On the one hand, since $\Delta \rj \p$, we get that $\maxh_\delta \subseteq H_\p$. On the other hand, 
Since the hypotheses in $\H$ are mutually exclusive and exhaustive we have that the set of all and only the hypotheses entailing $\neg \f$ equals $\H\setminus H_\f$. Hence, from  $\Delta \not\rj\neg\f$ it follows that $\maxh_{\Delta}\not \subseteq \H \setminus H_\f$, that is,  
 $\maxh_{\Delta} \cap H_\f \neq \emptyset$.
Now, we claim that  $\maxh_{\Delta,\f}$ is a subset of $\maxh_{\Delta}$.
From this, together with the premise $\Delta\rj\p$, it follows that, for every $H\in\maxh_{\Delta,\f}$ we have $H\models \p$. But this amounts at saying that $\Delta,\f\rj\p$, thus showing that $(\text{RMO})$ holds.}

{Let us thus prove the claim that $\maxh_{\Delta,\f}$ is a subset of $\maxh_{\Delta}$. Assume by contradiction that there is a $H \in \maxh_{\Delta,\f}$ such that $H\not \in  \maxh_{\Delta}$. Now, recalling that $\maxh_\Delta\cap H_\f \neq \emptyset$, pick $H'$ in such a set. By the fact that $H'\in H_\f$ and part (2) of Definition \ref{def:degrej}, we have $r_\f(H') =0$, hence $r_{\f}(H')\leq r_\f(H)$. On the other hand since $H'\in\maxh_{\Delta}$ and $H\not\in \maxh_{\Delta}$, we have $r_\Delta(H')<r_\Delta(H)$. But these two facts entail that $r_{\Delta,\f}(H')< r_{\Delta,\f}(H)$, contradicting the fact that $H\in \maxh_{\Delta,\f}$.}

For (CMO) assume that $\Delta \rj \psi$ and $\Delta\rj \delta$, i.e. that for any $H\in \maxh_{\Delta}$ we have both $H\models\psi$ and $H\models\delta$. 
We now show that $\maxh_{\Delta,\delta} \subseteq \maxh_{\Delta}$. 

Suppose that there is a $H\in \maxh_{\Delta,\delta}$ such that $H\not \in \maxh_{\Delta}$.  Thus, there is a $H'\in \maxh_{\Delta}$ such that $r_{\Delta}(H)>r_{\Delta}(H')$. On the other hand, since $H\in \maxh_{\Delta,\delta}$  we have that $r_{\Delta,\delta}(H) \leq r_{\Delta,\delta}(H')$, from which it follows that $r_{\Delta}(H) + r_\delta(H) \leq r_{\Delta}(H') +  r_\delta(H')$ and $r_{\Delta}(H) -r_{\Delta}(H')\leq  r_\delta(H') - r_\delta(H)$. 

Since $r_{\Delta}(H)>r_{\Delta}(H')$, we have that  $0< r_{\Delta}(H) -r_{\Delta}(H')$. Consequently $0< r_\delta(H') - r_\delta(H)$ and $r_\delta(H)<r_\delta(H')$.  Recall that $\delta\in\fm_\H$ and in view of the assumption that the hypotheses are mutually exclusive and exhaustive, either $T,H \models \neg \delta$ or $T,H \models\delta$. Hence we have that $r_\delta(H) >0$ iff $H\models\neg \delta$, otherwise $r_\delta(H)= 0$. Since $H\models \delta$, then $H\not \models\neg \delta$  and $r_{\delta}(H)=0$. Therefore, from $r_\delta(H)<r_\delta(H')$ it follows that $r_\delta(H')>0$. But $r_\delta(H') >0$ means in turn that $H'\models \neg \delta$ which contradicts the fact that $H'\in \maxh_{\Delta}$, in view of the assumption $\Delta\rj \delta$.  We have thus shown that $\maxh_{\Delta,\delta}\subseteq \maxh_{\Delta}$. From this, the conclusion immediately follows. For, if $\maxh_{\Delta,\delta} = \emptyset$, we immediately get $\Delta,\delta\rj\p$. Otherwise,  for any $H\in  \maxh_{\Delta,\delta}$, we also have $H\in \maxh_{\Delta}$, and by the assumption $\Delta\rj\p$, that $H\models \p$. This finally shows that $\Delta,\delta\rj \p$.
\end{proof}

\medskip

\Fourthprop*	
\begin{proof}
%For $\text{(AND)}_l$, it suffices to show that, for each $H\in \H$, we have $r^\a_{\c\land\d}(H)\leq r^\a_{\c}(H) + r^\a_{\d}(H)$.
%
%First, let $H_i\in H$ with $i\neq n$. 
%In case $P(t_H(\c\land\d)|H)>\a$ then we have $r_{\c\land\d}^\a(H)=0$. On the other hand, we will also have, by the monotonicity of probability and $t_H$, that $P(t_H(\c)|H)>\a$ and $P(t_H(\c)|H)>\a$, from which it follows that $r_{\c\land\d}^\a(H)\leq r_\c^\a(H) + r_\d^\a(H)=0$.
%
%Consider now the case $P(t_H(\c\land\d)|H)\leq \a$. We thus have 
%$r_{\c\land\d}^\a(H)=1-P(t_H(\c\land \d)|H)$.
%Now, assume 
Let $\H = \{H_1,H_2\}$ and $\c,\d\in \fm_\D$ such that $\neg\c,\neg \d\models\bot$. Let $\a_1=\a_2= 0.3$ and both $t_{H_1}$ and $t_{H_2}$ coincide with the identity function. Consider the following probability assignments: 
\begin{align*}
P(\c\land\d \mid H_1) = 0.2  & \quad P(\c\land\d \mid H_2 ) = 0.1  \\
P(\c\land\neg \d \mid H_1 ) = 0 & \quad  P(\c\land\neg \d \mid H_2) = 0.4 \\
 P(\neg\c\land\d\mid H_1 ) = 0.8 & \quad P(\neg\c\land\d \mid H_2) = 0.5 
\end{align*}

Since $\neg\c,\neg\d\models \bot$ we have 
\[P(\neg \c\land \neg \d \mid H_1 ) = P(\neg \c\land \neg\d \mid H_2 ) =0. \] 

Now, by simple computations, we obtain 
\begin{align*}
	P(\c \mid H_1) = 0.2 &  \quad  P(\c \mid H_2) = 0.5  \\
	P(\d \mid H_1) = 1  & \quad P(\d \mid H_2) = 0.6 \\
\end{align*}

By the definition of $r^\a$, we thus get:

\begin{align*}
	r^\a_\c(H_1) = 0.8  &\quad r_\c( H_2) = 0  \\
	r^\a_\d (H_1) = 0  &\quad r_\d( H_2) = 0 \\
   r^\a_{\c\land\d} (H_1) = 0.8  & \quad r^\a_{\c\land\d}( H_2) = 0.9 \\
\end{align*}
We thus obtain \[ r^\a_{\c,\d}(H_1) = r^\a_\c(H_1)+ r^\a_\d (H_1) = 0.8 \]
and 
\[ r^\a_{\c,\d}(H_2) = r^\a_\c(H_2)+ r^\a_\d (H_2) = 0 \]

Thus $\maxh_{\c,\d} = \{H_2\}$, while on the other hand, since 
$r^\a_{\c\land\d} (H_1)< r^\a_{\c\land\d}( H_2)$, we also have 
$\maxh_{\c\land\d} = \{H_1\}$.

This shows that 
\[ \c,\d\arj H_2 \qquad \c\land\d \not\arj H_2\]
thus proving that the $(\text{AND})_l$ rule does not hold for $\arj$.

The same example also shows that the converse $(\text{AND})_l^{con}$ does not hold, since:
\[ \c\land\d\arj H_1 \qquad \c,\d \not\arj H_1\]

%
%We find a counterexample, showing that for a certain $r^\a$, the premise of $\text{(AND)}_l$ is satisfied, while the conclusion is not.
%Fix $\a\in(0,1)$, and let $\H = \{H_1,H_2\}$. Suppose $P(t(\d)|H_1) > \a$ and  $P(t(\c)|H_1) > \a $ while $P(t(\d\land \c)|H_1) \leq \a$,  with $\d,\c\in \fm_\D$ such that $\neg \d,\neg \c \models\bot$ and  both $\d\not \models \neg H_2$ and $\c\not \models \neg H_2$.
% 
%We will thus have $r^\a_{\d,\c} (H_1)= r^\a_{\d}(H_1) + r^\a_{\c}(H_1) = 0 $  and $r^\a_{\d,\c} (H_2)= r^\a_{\d}(H_2) + r^\a_{\c}(H_2) =0$. Hence $\maxh_{\d,\c} = \{H_1,H_2\}$, so that  
%\[ \d,\c\rj H_1\lor H_2\]
%On the other hand, since  $P(t_{H_1}(\d\land \c)|H_1) \leq \a$, we get $r^\a_{\d\land\c} (H_1) = 1- P(t_{H_1}(\d\land \c)|H_1) $, while still $r^\a_{\d\land\c} (H_2) = 0$. Hence $\maxh_{\d\land\c} = \{H_2\}$ and 
%\[ \d\land \c\not\rj H_1\lor H_2\] 
\end{proof}

\medskip

\Fifthprop*	

\begin{proof}
For (REF), (CUT),  (CMO), (RMO) proceed as in Proposition \ref{prop:rjaddrules}.

{For $(\umon)$, assume that $\Delta\urj\p$ and $\{\Delta\setminus\{\delta\}\urj  \p\}_{\delta\in\Delta }$. By way of contradiction, let us assume $\Delta,\f{\not\urj}\p$, i.e. that $\maxh_{\Delta,\f}\neq \emptyset$ (and thus $r_{\Delta,\f}(H_{s^*})<m+1$) and in particular there is a $H\in \maxh_{\Delta,\f}$ such that $H\not\models\p$. 
Note that  $H\not\in\maxh_{\Delta}$ , since otherwise by $\Delta\urj \p$, we would immediately get $H\models\p$. }
%Now in case $\maxh_{\Delta} \subseteq \maxh_{\Delta,\f}$ we would immediatley obtain that H

{Let us pick now any  $H'\in \maxh_{\Delta}$, so that  $r_{\Delta}(H)> r_{\Delta}(H')$ and $H' \models \p$ (since $\Delta\urj\p$). Note that, 
since $H\in \maxh_{\Delta,\f}$ and $H'\in \maxh_{\Delta}$, both $r_{\Delta}(H)$ and $r_{\Delta}(H')$ are distinct from $\infty$.
%r_{\Delta,\f}(H)\leq r
On the other hand, since  $H\in \maxh_{\Delta,\f}$ we have 
$r_{\Delta,\f}(H) \leq r_{\Delta,\f}(H')$, hence $ r_{\Delta}(H) + r_\f(H) \leq r_{\Delta}(H')+ r_{\f}(H')$ and  
$r_\f(H) - r_\f(H') \leq r_\Delta(H') - r_\Delta(H) < 0$. Thus in particular we obtain $r_\f(H)< r_\f(H')$.  But  this can only occurr if $r_\f(H) = 0$ and $r_\f(H') = 1$.
And this means that $r_{\Delta,\f}(H) =r_\Delta(H)$ while  $r_{\Delta,\f}(H') = r_\Delta(H') +1$ and since $r_{\Delta,\f}(H)\leq r_{\Delta,\f}(H')$, we have  $r_\Delta(H)\leq r_\Delta(H')+1$. Hence we have both  $r_\Delta(H) > r_\Delta(H')$ and 
$r_\Delta(H)\leq r_\Delta(H')+1$, that is $r_\Delta(H) = r_\Delta(H') +1$. }

% Hence we may only have $r_{\Delta,\f}(H) = r_{\Delta,\f}(H')$, with $r_{\Delta}(H')  = r_{\Delta}(H)+1$ and $r_\f(H) = r_\f(H') -1$. Recall that $r_{\Delta}(H') = r_{\Delta}(H)+1$. 
{This means that there is a $\ol{\delta}\in \Delta$ rejecting $H$ and not rejecting $H'$. Let us consider $\Delta\setminus\{\ol{\delta}\}$. By our choice of $\ol{\delta}$, we will now have  $r_{\Delta\setminus\{\ol{\delta}\}}(H') = r_{\Delta\setminus\{\ol{\delta}\}}(H)$, hence $H'\in \maxh_{\Delta\setminus\{\ol{\delta}\}}$.
But this, by the assumption, \newline $\{\Delta\setminus\{\delta\}\urj  \p\}_{\delta\in\Delta }$, entails $H'\models\p$, which is the desired contradiction.}
\end{proof}

\medskip

\Firstlemma*	

\begin{proof}
Let $\Delta,\c\urj\p$ and $\Delta,\d\urj\p$. To show 
$\Delta,\c\lor\d\urj\p$, let $H\in\maxh_{\c\lor\d}$. If either $H\in \maxh_{\c}$ or $H\in\maxh_{\d}$, we immediately get, by the assumption, that $H\models\p$. We thus assume  $H\not\in \maxh_{\c}$ and $H\not\in \maxh_{\d}$, and derive a contradiction. By $H\not\in \maxh_{\d}$, we obtain that there exists a $H'$ such that
\[r_{\Delta,\c} (H') = r_{\Delta}(H') + r_\c(H')  < r_{\Delta}(H) + r_\c(H) = r_{\Delta,\c} (H)  \]
and by $H\not\in \maxh_{\p}$, analogously, we obtain that there exists a $H''$ such that 
\[r_{\Delta,\d} (H'') = r_{\Delta}(H'') + r_\d(H'')  < r_{\Delta}(H) + r_\d(H) = r_{\Delta,\d} (H)  \]
However, since $H\in\maxh_{\c\lor\d}$ we have:
\[ r_{\Delta,\c\lor\d}(H)  = r_{\Delta}(H) + r_{\c\lor\d}(H)\leq r_{\Delta}(H') + r_{\c\lor\d}(H') = r_{\Delta,\c\lor\d}(H')   \]
and 
\[ r_{\Delta,\c\lor\d}(H)  = r_{\Delta}(H) + r_{\c\lor\d}(H)\leq r_{\Delta}(H'') + r_{\c\lor\d}(H'') = r_{\Delta,\c\lor\d}(H'').  \]
{Now, note that $r_{\c\lor\d}(H) = 1$ if and only if $\c\lor\d \models\neg H$. Hence $r_{\c\lor\d}(H) = 1$  if $\c\models\neg H$ and $\d\models\neg H$, and it is 0 otherwise.} In other words,  $r_{\c\lor\d}(H) = \min (r_{\c}(H),r_{\d}(H))$, 
 and the same clearly holds for $H'$ and $H''$. Thus,  combining our previous inequalities, we obtain
\[r_{\Delta}(H) + r_{\c\lor\d}(H) \leq r_{\Delta}(H') + r_{\c\lor\d}(H') \leq r_{\Delta}(H') + r_{\c}(H') < r_{\Delta}(H) + r_{\c}(H)   \]
and similarly
\[r_{\Delta}(H) + r_{\c\lor\d}(H) \leq r_{\Delta}(H'') + r_{\c\lor\d}(H'') \leq r_{\Delta}(H'') + r_{\d}(H'') < r_{\Delta}(H) + r_{\d}(H)   \]
{But now, since $r_{\c\lor\d}(H) = \min (r_{\c}(H),r_{\d}(H))$ it is either equal to $r_{\c}(H)$ or to $r_{\d}(H)$, contradiction. }
\end{proof}

\medskip

\Secondlemma*	

\begin{proof} Assume the two premisses of MT: 
\begin{itemize}
\item[(A1)] $\Delta, \varphi \urj \psi$, i.e. for each $H \in \maxh_{\Delta, \varphi}$, $H \models \psi$, and that 
\item[(A2)]$\Delta \urj  \neg \psi$, i.e. for each $H \in \maxh_{\Delta}$, $H \models \neg \psi$. 
\end{itemize}
We show that for each $H \in \maxh_{\Delta}$, $H \models \neg \varphi$. 
Suppose, on the contrary, that there exists $H' \in \maxh_{\Delta}$ such that $H' \not \models \neg \varphi$, i.e. $\varphi$ does not reject $H'$.  
%Since $\varphi$ rejects $H$ if and only if $\varphi \models \neg H$, we have that $\varphi \not \models \neg H'$, i.e. 
Therefore $H' \in \maxh_{\Delta, \varphi}$ and $H' \models \psi$, but this is in contradiction with (A2).
\end{proof}

\medskip

\Firsttheorem*	
	
	\begin{proof}
We will show how to find a derivation of $\Delta' \urj \neg H_j$,  proceeding by induction on the number of formulas occurring in $\Delta'$. 
Henceforth, in the applications of 
$(\umon)$ and $(\rmon)$ we will not explicitly consider the additional condition  that $r_{\Delta,\f}(H)$ is finite for each $H\in\H$ since this will always hold, under our assumption that $r_\Delta'(H)$ is finite for each $H$.	
	
%Lemma \ref{normalform} guarantees the existence of a set of
%\[\neg H_1^{l_1}, \dots, \neg H_n^{l_n} \urj \neg H_j,\] 
%one for every $H$ s.t. $H \models \psi$, where all the multiplicity indices $l_1, \dots, l_n$  are greater or equal than $0$ and $j\in \{1,\dots,n\}$.
%Pick any such consequences and denote its left-hand side by $\Delta'$.
% We will show how to find a derivation of $\Delta' \urj \neg H_j$. Henceforth, in the applications of 
%$(\umon)$ and $(\rmon)$ we will not explicitly consider the additional condition  that $r_{\Delta,\f}(H)$ is finite for each $H\in\H$ since this will always hold, under our assumption that $r_\Delta'(H)$ is finite for each $H$.
%%\paolo{$r_{\Delta,\f}(H_{s^*})\leq m$ and $r_{\Delta,\f}(H_i)\leq m+1$ for $H_i\neq H_{s^*}$ }, since this will always hold, under our assumption that $\Delta' \urj \neg H_j$.
%

For the base case, when $\Delta'$ is composed of only one formula, it has to be of the form $\neg H_j\urj \neg H_j$, which is clearly an instance of (REF). 
That this can be the only case when the size of $\Delta'$ is one, follows from the fact that, for any $i\neq j$ we would have that   $H_j\in \maxh_{\neg H_i}$ and $H_j\not \models \neg H_j$. Hence $\neg H_i\not \urj \neg H_j$ if $i\neq j$.
 
Now let us consider the inductive step. By Definition \ref{def:urj} and the normal form of the formulas in $\Delta'$ we have two possibilities:
\begin{enumerate}
\item  $\maxh_{\Delta'} = \{H\mid \neg H \not \in \Delta'\}$ that is, the least rejected hypotheses are those that do not occur in $\Delta'$ at all, or
  \item $\maxh_{\Delta'} =\{ H_k \mid l_k\neq 0 \text{ and } l_k$ is a minimal non-zero index among the $l_1,\dots, l_n\}.$
\end{enumerate}
 
%
%\[\maxh_{\Delta'} \subseteq  \{ H \mid \neg H \in \Delta'   \}. \]

In case (1), for any $H\in\maxh_{\Delta'}$, we have $\neg H \not \in\Delta'$, i.e. $r_{\Delta'}(H) = 0$.  However, we have that $H_j\not \in \maxh_{\Delta'}$. Indeed, if we had $H_j \in \maxh_{\Delta'}$, from $\Delta' \urj \neg H_j$, we would get that $  H_j\models\neg H_j$, which is absurd.

Now, if there are no hypotheses in $\Delta'$ distinct from $H_j$, this means that our consequence is of the form $(\neg H_j)^{l_j} \urj \neg H_j$. Hence we could reduce its size by the following application of $\rmon$.
\[ \infer[(\rmon)]{(\neg H_j)^{l_j}\urj\neg H_j}{\neg H_j^{l_j-1} \urj \neg H_j & (\neg H_j)^{l_j-1} \not \urj \neg \neg H_j } \]

Now upon removal of any formula $\neg H'$ from $\Delta'$ distinct from $\neg H_j$, we will still have $r_{\Delta'\setminus\{\neg H'\}}(H) = 0$, for any $H\in\maxh_{\Delta'}$. This means that, if $H \in\maxh_{\Delta'}$, then still $H\in\maxh_{\Delta'\setminus\{\neg H'\}}$. Recall now that $T_\Delta, H\models \neg H'$ for any $\neg H'\in \Delta'$, and $H\models\neg H_j$. 

On the other hand, by assumption, $\maxh_{\Delta'}$ does not contain $H_j$, and since $H'\neq H_j$, we may safely assume that $H_j \not\in \maxh_{\Delta'\setminus\{\neg H'\}}$.  This means that still $\Delta'\setminus\{\neg H'\}\urj H_j$.

Hence,  we  have shown that both 
$\Delta'\setminus\{\neg H'\}\not\urj H'$ and $\Delta'\setminus\{\neg H'\}\urj H_j$.  We then get $\Delta'\urj\neg H_j$ by applying {(\rmon)} as follows:
\[ \infer[({\rmon)}]{\Delta'\urj\neg H_j}{\Delta'\setminus \{\neg H'\} \urj \neg H_j & \Delta'\setminus \{\neg H'\} \not\urj \neg\neg H' }.\]
Let us now consider case (2), where the least rejected hypotheses occurr (negated) in $\Delta'$, i.e. for any $H\in\maxh_{\Delta'}$, we have $\neg H\in \Delta'$, hence $r_{\Delta'}(H) \neq 0$. 
Since $H_j \not \models \neg H_j$ and $\Delta'\urj\neg H_j$ we may again exclude that $H_j$ is any of the least rejected hypotheses.

Let us first consider the subcase (2a) where $H_j$ is the maximally rejected hypothesis. We need to distinguish three sub(sub)cases.

(2a') If in $\Delta'$ there is another maximally rejected hypothesis, say $H_p$, in addition to $H_j$, we use 
(\rmon) to remove (reasoning backwards) one of its occurrences.

First, note that since $\Delta'\urj \neg H_j$, we have, for any $H\in \maxh_{\Delta'}$, that $H\models\neg H_j$, and in particular $H_j \not \in \maxh_{\Delta'}$.  The removal of an occurrence of a maximally rejected $\neg H_p$ will not affect this. In other words, we still have $H_j\not \in \maxh_{\Delta'\setminus \{\neg H_p\}}$ hence ${\Delta'\setminus\{\neg H_p\}\urj\neg H_j}$.
On the other hand, recall that  by assumption of case (2) the least rejected hypotheses, say $H_k,$ occur in $\Delta'$. Upon removal of an occurrence of the maximally rejected hypothesis $H_p$, we still have $H_k \in \maxh_{\Delta' \setminus\{\neg H_p\}}$. Since $H_k\models \neg H_p$, we will then have $\Delta' \setminus\{\neg H_p\}\not\urj H_p$. 
Hence we may derive $\Delta'\urj \neg H_j$ as follows:
\[\infer[{(\rmon)}]{\Delta'\urj\neg H_j}{\Delta'\setminus\{\neg H_p\}\urj\neg H_j & \Delta'\setminus\{\neg H_p\} \not\urj \neg\neg H_p}.\]

(2a'') Assume now that: (i) $H_j$ is a maximally rejected hypothesis (as per assumption of case (2a));  (ii) $H_j$ is the unique maximally rejected hypothesis (in contrast to case (2a')); (iii) $H_k$ is a least rejected hypothesis, that by assumption (2) occurs (negated) in $\Delta'$. 

This means that we have $r_{\Delta'}(H_j)>r_{\Delta'}(H_k)$.
Note that $r_{\Delta'}(H_j)= r_{\Delta'\setminus\{\neg H_k\}}(H_j)  $ and $r_{\Delta'}(H_k) = r_{\Delta'\setminus\{\neg H_k\}}(H_k) +1$, hence
\begin{equation}\tag{*}\label{eq:*}
\ r_{\Delta'\setminus\{\neg H_k\}}(H_j) > r_{\Delta'\setminus\{\neg H_k\}}(H_k) +1 > r_{\Delta'\setminus\{\neg H_k\}}(H_k).
\end{equation}
This entails  $H_j\not\in \maxh_{\Delta'\setminus\{\neg H_k\}}$, hence
$\Delta'\setminus{\{\neg H_k\}}\urj \neg H_j$. However, for any $\neg H \in \Delta'\setminus\{H_k\}$, we have  either
$$ r_{\Delta'\setminus\{\neg H_k,\neg H\}}(H_j) =  r_{\Delta'\setminus\{\neg H_k\}}(H_j) - 1$$
or
$$r_{\Delta'\setminus\{\neg H_k,\neg H\}}(H_j) =  r_{\Delta'\setminus\{\neg H_k\}}(H_j)$$
depending on whether $\neg H \models \neg H_j$ (i.e. $H$ is actually $H_j$)  or not. 
In both cases we have
\[r_{\Delta'\setminus\{\neg H_k,\neg H\}}(H_j) \geq  r_{\Delta'\setminus\{\neg H_k\}}(H_j) - 1.\]
Hence we obtain
\[ r_{\Delta'\setminus\{\neg H_k,\neg H\}}(H_j) \geq r_{\Delta'\setminus\{\neg H_k\}}(H_j) - 1  > 
r_{\Delta'\setminus\{\neg H_k\}}(H_k) \geq r_{\Delta'\setminus\{\neg H_k,\neg H\}}(H_k) \]
where the strict inequality follows from \eqref{eq:*}, and the last inequality follows directly from the definition of the degree of rejection $r$.

%
%In the other case, we still have: 
%\[r_{\Delta'\setminus\{\neg H_k,\neg H\}}(H_j) =  r_{\Delta'\setminus\{\neg H_k\}}(H_j) > 
%r_{\Delta'\setminus\{\neg H_k\}}(H_k) \geq r_{\Delta'\setminus\{\neg H_k,\neg H\}}(H_k)\]

This shows that $H_j \not \in\maxh_{\Delta'\setminus\{\neg H_k,\neg H\}} $ for any choice of $\neg H$ (including $\neg H_j$). 
%, we need also have $H_k \in \maxh_{\Delta'\setminus\{\neg H_k\}} $, and since $H_k\models \neg H_j$ this means that 
Hence, for any $\neg H\in \Delta'$:  \[\Delta'\setminus\{\neg H_k,\neg H\}\urj  \neg H_j. \]
%since $H_j$ is a maximally rejected hypothesis of $\Delta'$,  upon removal of $\neg H_k$, we will still have $s(H_j) s(H_k)$.
We can thus derive $\Delta'\urj \neg H_j$ using the rule application:

\[\infer[(\umon)]{\Delta'\urj\neg H_j}{\Delta'\setminus{\{\neg H_k\}}\urj \neg H_j &  \{\Delta'\setminus\{\neg H_k,\neg H\}\urj  \neg H_j\}_{\neg H \in\Delta' } }    \]
and the above reasoning, showing that its premises hold.

%We we repeat the procedure until we either obtain $\neg H_j^{l_j} \urj \neg H_j$ 

%We are left to show that $\Delta' \urj \neg H_j $ is derivable when other negated hypothesis are present in $\Delta'$.
%We need to consider again two subcases.

Finally, our last subcase of $(2a)$ is $(2a''')$,  when there is no other hypotheses but $\neg H_j$ in $\Delta'$, i.e. $\Delta'\urj\neg H_j$  is actually of the form  ${\neg H_j^{l_j} \urj \neg H_j}$. We derive it, starting from $\neg H_j\urj \neg H_j$, by repeated applications of (\rmon), beginning with  
\[\infer[(\rmon)]{\neg H_j,\neg H_j\urj\neg H_j}{\neg H_j \urj \neg H_j & \neg H_j \not\urj \neg\neg H_j}.\]

Consider now case (2b), where $H_j$ is not a maximally rejected hypotheses, and 
let $H_p$ be any maximally rejected hypothesis. Since 
$\Delta'\urj \neg H_j$, we know that $H_j\not \in \maxh_{\Delta'}$, i.e. $H_j$ is not a least rejected hypothesis either. Upon removal of one occurrence of $\neg H_p$, $H_j$ will still not be a least rejected hypothesis, hence 
 $H_j\not \in \maxh_{\Delta'\setminus \{\neg H_p \}}$and  ${\Delta'\setminus \{\neg H_p \}\urj \neg H_j}.$ 
Moreover, there will be at least a least rejected hypothesis in $ \maxh_{\Delta'\setminus \{\neg H_p \}}$  that is distinct from $H_p$, hence  $\Delta'\setminus \{\neg H_p \} \not\urj H_p$.

This means that we may consider the following application of {(\rmon)}:

\[ \infer[{(\rmon)}]{\Delta'\urj \neg H_j} {\Delta'\setminus \{\neg H_p \}\urj \neg H_j  & \Delta'\setminus \{\neg H_p \} \not\urj \neg\neg H_p  }, \]
completing our proof.\end{proof}

\bibliographystyle{plain}

\end{document}